\documentclass[onecolumn]{autart}    

\usepackage{graphicx}           
\usepackage{dcolumn}            
\usepackage{bm}                 

\usepackage{graphics}           
\usepackage{epsfig}             
\usepackage{times}              
\usepackage[fleqn]{amsmath}
\usepackage{mathrsfs}
\usepackage{amssymb}
\usepackage{extarrows}

\usepackage{ntheorem} 

\usepackage{amsfonts}
\usepackage{fancyhdr}
\usepackage{color}
\usepackage{subfigure}
\usepackage{enumerate}
\usepackage{url}

\allowdisplaybreaks[4] 
\newtheorem{theorem}{Theorem} 

\newtheorem{proposition}{Proposition}
\newtheorem{definition}{Definition}

\newtheorem*{proof}{Proof}

\newcommand{\diff}[1]{\text{d}{#1}}

\begin{document}

\begin{frontmatter}

\title{A PDE--based Power Tracking Control of Heterogeneous TCL Populations} 

\author{Jun Zheng$^{1}$}\ead{zhengjun@swjtu.edu.cn},
\author{Guchuan~Zhu$^{2}$}\ead{guchuan.zhu@polymtl.ca},
\author{Meng~Li$^{3}$}\ead{lmbuaa@gmail.com}

\address{$^{1}${School of Mathematics, Southwest Jiaotong University,
        Chengdu, Sichuan, China 611756}\\
        $^{2}$Department of Electrical Engineering, Polytechnique Montr\'{e}al,  P.~O. Box 6079, Station Centre-Ville, Montreal, QC, Canada H3T 1J4 \\
        $^{3}$College of Big Data and Internet, Shenzhen Technology University, Shenzhen, Guangdong, China 518118}

\begin{abstract}                           
This chapter presents the development and the analysis of a scheme for aggregate power tracking control of heterogeneous populations of thermostatically controlled loads (TCLs) based on partial differential equations (PDEs) control theory and techniques. By employing a thermostat--based deadband control with forced switching in the operation of individual TCLs, the aggregated dynamics of TCL populations are governed by a pair of Fokker--Planck equations coupled via the actions located both on the boundaries and in the domain. The technique of input-output feedback linearization is used for the design of aggregate power tracking control, which results in a nonlinear system in a closed loop. As the considered setting is a problem with time-varying boundaries, well-posedness assessment and stability analysis are carried out to confirm the validity of the developed control scheme. A simulation study is contacted to evaluate the effectiveness and the performance of the proposed approach.
\end{abstract}

\end{frontmatter}
%
\section{Introduction}\label{Sec: Introduction}
Control of large populations of thermostatically controlled loads (TCLs), such as refrigerators, air conditioners, space heaters, hot water tanks, etc., has drawn a considerable attention in the recent literature, see, just to cite a few, \cite{Angeli:2012,Callaway:2009,laparra2019,Liu:2016,Liu2016MPC,Mahdavi2017,tindemans2015decentralized,Totu:2017,Zhang2013Aggregated}. The development in this field is mainly motivated by the fact that most of the TCLs exhibit flexibilities in power demand for their operation and elasticities in terms of performance restrictions. Therefore, a large scale TCL population can be managed as a Demand Response (DR) resource to provide such features as power peak shaving and valley filling, as well as absorbing fluctuations of renewable energies and enabling dynamic pricing schemes in the context of the Smart Grid \cite{deng2015survey,Palensky2011}.

Most of the TCL population control techniques developed in recent years can be considered as extensions of the traditional method of direct load control (DLC) by exploiting the bidirectional communication capability enabled by the smart grid paradigm. This is also one of the basic assumptions required in the present work. Moreover, thermostatic--based deadband control is one of the most used schemes in the operation of TCLs, which is convenient for supporting two basic types of switching control, namely fast commutation for power control and slow energy consumption regulation. A fast commutation between ON and OFF states can achieve instantaneous power consumption controls, which can provide, for example, auxiliary services, such as frequency control and load following \cite{Liu:2016,Liu2016MPC,Mahdavi2017,Zhang2013Aggregated}. The long-term energy consumption can be regulated through set-point control, deadband width variation, switching duty cycle adjustment, etc., so that a large population of TCLs can be managed to provide the DR capability \cite{Angeli:2012,Bashash2013,Callaway:2009,Soudjani2015,Totu:2017}.

The focus of this work is put on aggregate power control of heterogeneous TCL populations based on the dynamic model of the population, which is in general composed of the microscopic dynamics of individual TCLs and the macroscopic aggregate model of the population. At the level of individual TCLs, the equivalent thermal parameter (ETP) model is the most adopted one for describing the temperature dynamics of TCLs (see, e.g., \cite{Koch2011Modeling,Liu:2016,Radaideh2019,Zhang2013Aggregated}). At the aggregate level, a pioneer work on the modeling of the aggregated dynamics of TCL populations can be traced back to the one reported in \cite{Malhame:1985}, where a continuum model described by partial differential equations (PDEs) was introduced for the first time. Specifically, it is shown that under the assumption of {homogeneous} TCL populations with all the TCLs modeled by thermostat-controlled scalar stochastic differential equations, the dynamics of a population can be expressed by two coupled Fokker--Planck equations describing the evolution of the distribution of the TCLs in ON and OFF states, respectively. The PDE--based aggregate model has drawn much attention in the recent literature for different DR applications \cite{Angeli:2012,Bashash2013,Callaway:2009,Ghaffari:2015,laparra2019,moura2013modeling,tindemans2015decentralized,Totu:2017,ZLZL:2019} and has led to mathematically rigorous results obtained by leveraging the rich theory from PDEs. This continuum perspective has also enabled more abstract formulations that are independent of the particular solutions. A more generic stochastic hybrid system model applicable to a wider class of responsive loads is developed in~\cite{Zhao:2018}. It should be noted that the diffusive term in the Fokker--Planck equations can capture the heterogeneity of the population at a certain level \cite{moura2013modeling}. Whereas, a heterogeneous population with high diversity can be divided into a finite number of groups, each of which represents a population with limited variation in its heterogeneity~\cite{Ghaffari:2015}.

Another popular and widely adopted model for describing the aggregated dynamics of TCL populations is based on the Markov chains or state bins representation. The original state-bin model proposed in~\cite{Koch2011Modeling} is described by a finite-dimensional linear time-invariant (LTI) system, which has been further extended to adapt to TCLs of more generic dynamics, e.g., higher order systems, with different control mechanisms, and to cope with parameter variations (e.g., time-varying ambient temperature) \cite{Hao2015Aggregated,Liu:2016,Liu2016MPC,Mahdavi2017,Radaideh2019,Soudjani2015,Zhang2013Aggregated}.
It is interesting to note that, as pointed out in~\cite{Zhao:2018}, the state-bin model can be seen as the discretization of a PDE--based model over a range of temperature. However, as most of the PDE aggregate models, such as the Fokker--Planck equation, are nonlinear, or semi-linear with a more accurate terminology in the PDE theory, and time-varying, a PDE model has also to be linearized around an operational point, e.g., the set-point corresponding to a reference temperature, with fixed parameters in order to obtain {an LTI system}. From this point of view, the PDE--based continuum model is indeed a more generic representation of the aggregated dynamics of TCL populations, which can incorporate nonlinearity, time-varying parameters, as well as higher-order ETP models \cite{Zhao:2018}, within a unified formulation.

Indeed, discretizing, and eventually linearizing, the PDE model to obtain a lumped finite-dimensional linear system, also referred as \emph{early-lumping}, is an approach widely adopted in the control of TCL populations \cite{Angeli:2012,Bashash2013,Callaway:2009,Ghaffari:2015,Totu:2017}. With this method, the complexity in control design and implementation, as well as the expected performance, should be similar to that carried out with state-bin--based models. Another approach, which allows better taking advantage of PDE--based models, is to perform the control design directly with the original PDE systems. This approach, referred as \emph{late-lumping}, can preserve the basic property of the original system, in particular the closed-loop stability, and the performance when the designed control is applied \cite{Balas:1978}. Another attractive feature of this approach is that the control implementation is computationally tractable and does not suffer from the diminution of discretization step-size, which is often needed for performance improvement. The application of the late-lumping method to the control of TCL populations has been demonstrated recently in \cite{Ghanavati:2018,ZLZL:2019}.

In the control design of this work, we adopt the approach of input-output linearization for distributed parameter systems \cite{Christofides1996,Christofides2001,Maidi:2014,ZLZL:2019}, which can avoid discretizing and locally linearizing the model by dealing directly with the original nonlinear PDEs, i.e., the Fokker--Planck equations, while coping with time-varying parameters, such as set-point and ambient temperature in a straightforward manner. Specifically, we choose first the aggregate power of the population as the output of the system. We then design a PDE control that renders the input-output dynamics to be a finite dimensional LTI system (see, e.g., \cite{Christofides1996,Christofides2001,Maidi:2014}). Note that this approach is inherited from the well-established technique of input-output linearization for finite dimensional nonlinear system control \cite{Isidori:1995,Khalil:2002,Krstic:1995,Levine:2009}. The developed control scheme is validated through the application to a typical scenario considered in many work reported in the recent literature (e.g., \cite{Angeli:2012,Bashash2013,Callaway:2009,Ghaffari:2015,Ghanavati:2018,Hao2015Aggregated,Koch2011Modeling,laparra2019,Liu:2016,Liu2016MPC,Mahdavi2017,moura2013modeling,Radaideh2019,Soudjani2015,tindemans2015decentralized,Totu:2017,Zhang2013Aggregated,ZLZL:2019}), namely aggregate power control of a large population of heating/cooling appliances that may spread over a wide area.

With the thermostat--based deadband control considered in the present work, switchings may happen both in the domain and on the boundaries. Furthermore, the boundaries are time-varying, which represents a more generic setting. It should be noted that in addition to the appearance of nonlinear nonlocal terms in the closed-loop system, the nature of time-varying boundaries will make control design, well-posedness assessment, and stability analysis more challenging. Nevertheless, we found that this difficulty could be overcome by transforming the original system into a one with fixed boundaries and hence, the techniques developed in \cite{ZLZL:2019} can be applied in well-posedness assessment and stability analysis of the problem considered in this work.
Finally, it is interesting to note that the considered setting covers the case where a number of TCLs may move beyond the current dead-band boundaries
when set-points are changed rapidly and hence, they are forced to switch in an unpredictable manner. In this chapter, we will present a rigorous well-posedness and stability analysis, which can provide a theoretical tool for solving this open problem, which cannot be handled by the existing methods.
Dealing with a generic setting of TCL populations commonly adopted in theoretical researches and practical applications while carrying out control design and analysis in a rigorous manner constitutes a main contribution of this work to the related literatures.

In the rest of this chapter, Section~\ref{Sec: Modeling} presents the dynamical model of individual TCLs under thermostat--based deadband control and the PDE aggregate model of heterogeneous TCLs populations with moving boundary conditions. Control design is detailed in Section~\ref{Sec: Control design}. Well-posedness analysis and stability assessment of the closed-loop system are presented in Section~\ref{Sec: Well-posedness} and Section~\ref{Sec: Stability Analysis}, respectively. A simulation study for evaluating the performance of the developed control scheme is carried out in Section~\ref{Sec: Simulation Study}, followed by some concluding remarks provided in Section~\ref{Sec: Concluding Remarks}.
Notations on function spaces and details of certain technical development are given in appendices.

\section{Modeling aggregate of TCL populations}\label{Sec: Modeling}
\subsection{Dynamics of TCLs under thermostat--based deadband control with forced switching}\label{Sec: Individual Dynamics of TCLs}
We consider a large population of TCLs operated in ON/OFF mode. Denote by $x_i$ the temperature of the $i$-th load in a population of size $N$ whose dynamics are described by (see, e.g., \cite{Bashash2013,Callaway:2009,Koch2011Modeling})
\begin{equation*}
  \dfrac{\diff x_i}{\diff t} = \dfrac{1}{R_i C_i}\left(x_{ie}-x_i - s_iR_iP_i\right), \ i=1,\ldots,N,
\end{equation*}
where $x_{ie}$ is the ambient temperature, $R_i$ is thermal resistance, $C_i$ is thermal capacitance, $P_i$ is thermal power, and the control signal $s_i(t)$ takes a binary value from $\{0,1\}$, representing OFF and ON states, respectively. Note that considering cooling systems is only for the sake of notational simplicity.

As shown in Fig.~\ref{Fig: random sw}, there are two types of switchings in a generic setting of thermostat--based control. In the regular operation regime, switchings will occur when the temperature reaches the boundaries of the deadband $(\underline{x}, \overline{x})$, where $\underline{x}$ and $\overline{x}$ are, respectively, the lower and upper temperature bounds. Forced switchings can be applied inside the deadband for fast power control, as mentioned earlier. Forced switchings can be randomly generated in order to, e.g., avoid the possible power demand oscillations due to synchronization \cite{Callaway:2009,laparra2019,Totu:2017,ZLZL:2019}. Spontaneous interrupts of the operational state of TCLs from the participants of a DR program for any particular motivations can also be classed as forced switchings.

\begin{figure}[htpb]\
  \centering
  \includegraphics[scale=0.225]{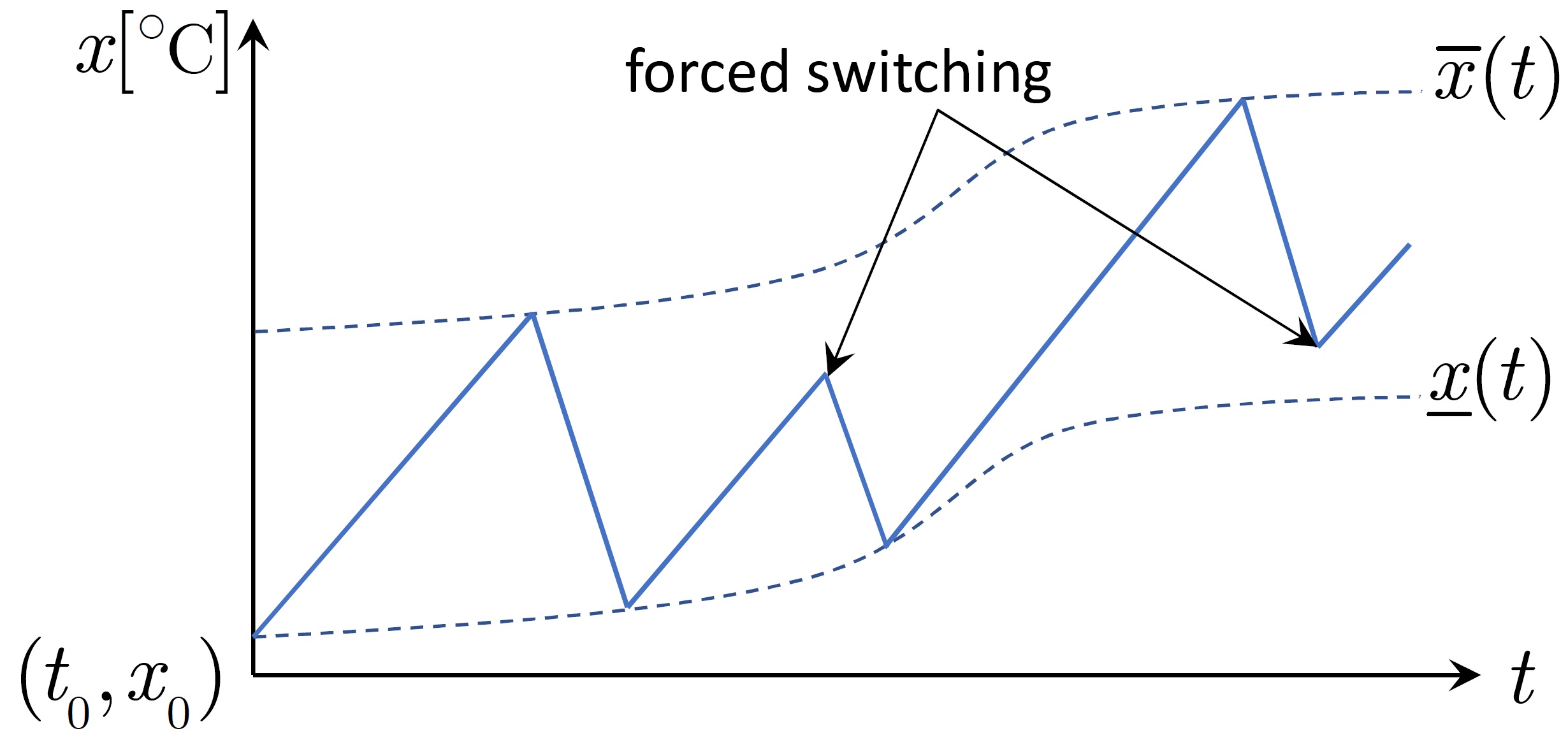}
  \caption{Temperature profile starting from a reference point $(t_0,x_0)$ for a TCL under thermostat--based deadband control with time-varying boundaries and forced switching.
           }\label{Fig: random sw}
\end{figure}

The {forced} switching signal for the $i$-th TCL, denoted by $r_i(t)$, takes also a binary value from $\{0,1\}$, with 1 representing the occurrence of switching and 0 otherwise. The control signal of the $i$-th TCL integrating the two types of switching actions can then be expressed as
\begin{equation*}
  s_i(t) =
  \begin{cases}
    1, & \hbox{if } x \geq \overline{x}; \\
    0, & \hbox{if } x \leq \underline{x};\\
    (s_i(t^-)\wedge r_i(t)) + (s_i(t^-)\vee r_i(t)), & \hbox{otherwise};
  \end{cases}
\end{equation*}
where ``$\wedge$'' and ``$\vee$'' represent the Boolean operations \textsf{AND} and \textsf{OR}, respectively, and ``$+$'' is the one-bit binary addition with overflow.

\subsection{PDE aggregate model}\label{Sec: PDE Model}
Throughout the chapter, let
$\mathbb{R}$, $\mathbb{R}_{\geq 0}$, $\mathbb{R}_{+}$ denote the set of all real numbers, nonnegative real numbers, and positive real numbers, respectively. For $T\in \mathbb{R}_{+}$, let $Q_T=(0,1)\times(0,T)$, $\overline{Q}_T=[0,1]\times[0,T]$, ${Q}_\infty=(0,1)\times\mathbb{R}_{+}$ and $\overline{Q}_\infty=[0,1]\times\mathbb{R}_{\geq 0}$. For $t\in[0,T]$, let ${\Omega}_T=(\underline{x}(t),\overline{x}(t))\times (0,T)$ and $\overline{\Omega}_T=[\underline{x}(t),\overline{x}(t)]\times [0,T]$. For $t\in \mathbb{R}_{+}$, let ${\Omega}_\infty=(\underline{x}(t),\overline{x}(t))\times \mathbb{R}_{+}$ and $\overline{\Omega}_\infty=[\underline{x}(t),\overline{x}(t)]\times \mathbb{R}_{\geq 0}$. {For a nonnegative integer $m$}, we denote by {$w_{\underbrace{s\cdots s}_{m \text{~times}}}$ the $m$th-order derivative} of a function $w$ w.r.t. its argument $s$.

Let {$\alpha_j = \frac{1}{R C}\left(x_{e}-x - s_jRP\right), \ j = 0, 1$,} where $R$, $C$, and $P$ are, respectively, the representative thermal resistance, thermal capacitance, and thermal power of the population, and the control signals $s_0(t) = 0$ and $s_1(t) = 1$ represent OFF and ON states. We recall that $x_{e}$ is the time-varying ambient temperature.

As mentioned earlier, power consumption control of a TCL population under deadband control can be achieved by moving the lower and upper temperature boundaries \cite{Bashash2013,Callaway:2009}. The width of the deadband is denoted by $\Delta x=\overline{x}- \underline{x}$. Fixing $\Delta x$, the control input can be chosen as $u(t)=\dot{\overline{x}}=\dot{\underline{x}}$, where $\dot{x}:=\frac{\diff x}{\diff t}$. Denote by $\delta_{0 \rightarrow 1}(x,t)$ (respectively $\delta_{1 \rightarrow 0}(x,t)$) the {net flow} due to the switch of loads at $(x,t)$ from OFF-state to ON-state (respectively from ON-state to OFF-state) {inside of the temperature bounds}. Under the assumption of mass-conservative, that is the size of the population is constant, there must be $\delta_{0 \rightarrow 1}(x,t) = -\delta_{1 \rightarrow 0}(x,t) := \delta(x,t)$. Let $w(x,t)$ and $v(x,t)$ be the distribution of loads [number of loads/$^{\circ}$C] at temperature $x$ and time $t$, over the ON and OFF states, respectively. Then, the dynamics of distribution evolution of a heterogeneous TCL population can be expressed by the following forced Fokker--Planck equations \cite{Bashash2013,Callaway:2009}:
\begin{subequations}\label{eq: foced FPE}
\begin{align}
  w_t(x,t) = \left( \beta w_x(x,t)- \left(\alpha_{1}(x)-u(t)\right)w(x,t)\right)_x+\delta(x,t),~&(x,t)\in \Omega_{\infty}, \label{eq: foced FPE ON}\\
  v_t(x,t) = \left( \beta v_x(x,t)- \left(\alpha_{0}(x)-u(t)\right)v(x,t)\right)_x-\delta(x,t),~&(x,t)\in \Omega_{\infty},\label{eq: foced FPE OFF}
\end{align}
\end{subequations}
with $\beta$ denoting the diffusion coefficient, satisfying the following boundary conditions:
\begin{subequations} \label{eq: BC 2}
\begin{align}
  {\mathscr{B}_1[w](\overline{x},t):=}\beta w_x(\overline{x},t) - \left(\alpha_{1}(\overline{x})
  {- u(t) -\dot{\overline{x}}} \right)w(\overline{x},t) = {\overline{\sigma} (t)},~&t\in\mathbb{R}_{\geq 0}, \\
  {\mathscr{B}_1[w](\underline{x},t):=}\beta w_x(\underline{x},t) - \left(\alpha_{1}(\underline{x})
  {- u(t) -\dot{\underline{x}}}\right)w(\underline{x},t) = {\underline{\sigma}(t)},~&t\in\mathbb{R}_{\geq 0}, \\
  {\mathscr{B}_2[v](\overline{x},t):=}\beta v_x(\overline{x},t) - \left(\alpha_{0}(\overline{x})
  {- u(t) -\dot{\overline{x}}}\right)v(\overline{x},t) = {-\overline{\sigma} (t)},~&t\in\mathbb{R}_{\geq 0}, \\
  {\mathscr{B}_2[v](\underline{x},t):=}\beta v_x(\underline{x},t) - \left(\alpha_{0}(\underline{x})
  {- u(t) -\dot{\underline{x}}}\right)v(\underline{x},t) = {-\underline{\sigma}(t)},~&t\in\mathbb{R}_{\geq 0},
 \end{align}
\end{subequations}
where $\overline{\sigma}(t)$ and $\underline{\sigma}(t)$ represent the net flux due to the switch of loads on the boundaries $\overline{x}$ and $\underline{x}$, respectively, including the flux due to the switch of loads beyond the current dead-band boundaries when set-points are changed rapidly.

The initial value conditions are given by
\begin{align}
w(x,0) = w^0(x), \ v(x,0) = v^0(x),~x\in (\underline{x}(0),\overline{x}(0)).\label{initialwv}
\end{align}
Clearly, the boundary conditions given in \eqref{eq: BC 2} describe the coupling of the $w$- and $v$-dynamics on the boundaries. Moreover, as shown in Proposition~\ref{prop. 8}, the property of masse conservation can be assured, i.e., for any $t\in\mathbb{R}_{\geq 0}$, it holds
\begin{align}\label{agg}
  N_{\text{agg}}(t) &= \int_{\underline{x}(t)}^{\overline{x}(t)} \left(w(x,t)+v(x,t)\right)\diff x
                    = \int_{\underline{x}(0)}^{\overline{x}(0)} \left(w^0(x)+v^0(x)\right)\diff x,
\end{align}
which represents the size of the population.

\subsection{Equivalent PDE model}
Note that the system defined in \eqref{eq: foced FPE} and \eqref{eq: BC 2} is a problem with moving boundaries for which it will be difficult to perform control system design and analysis. For this reason, we transform this system into an equivalent model with fixed boundaries. Specifically, let $z=\frac{x-\underline{x}}{\Delta x}$, then $\underline{x}(t)\leq x\leq\overline{x}(t)$ makes $0\leq z\leq 1$.
Thus, $w(x,t)=w(z\Delta x+\underline{x},t):=\widetilde{w}(z,t)$, $\delta(x,t)=\delta(z\Delta x+\underline{x},t):=\widetilde{\delta}(z,t)$
 and $\alpha_j(x,t)=\alpha_j(z\Delta x+\underline{x},t):=\widetilde{\alpha}_j(z,t),j=0,1$, which is expressed by
\begin{equation*}
  \widetilde{\alpha}_j(z,t) = \dfrac{1}{R C}\left((z_{e}-z)\Delta x - s_jRP\right), \ j = 0, 1,
\end{equation*}
with $z_e = \frac{x_e-\underline{x}}{\Delta x}$. Moreover,
\begin{align*}
  w_x=\frac{1}{\Delta x}\widetilde{w}_z,w_{xx}=\frac{1}{(\Delta x)^2}\widetilde{w}_{zz},{w_t=\widetilde{w}_t-\widetilde{w}_z\times \frac{\dot{\underline{x}}}{\Delta x}}.
  \end{align*}
Then the system given in \eqref{eq: foced FPE} becomes
\begin{align*}
  \widetilde{w}_t = \frac{1}{\Delta x}\left( \frac{\beta}{\Delta x}\widetilde{w}_z- \left(\widetilde{\alpha}_{1}-u-\dot{\underline{x}}\right)\widetilde{w}\right)_z
                                         +\widetilde{\delta},~&(z,t)\in Q_\infty,\\
  \widetilde{v}_t = \frac{1}{\Delta x}\left(\frac{\beta}{\Delta x}  \widetilde{v}_z- \left(\widetilde{\alpha}_{0}-u-\dot{\underline{x}}\right)\widetilde{v}\right)_z
                                         -\widetilde{\delta},~&(z,t)\in Q_\infty.
\end{align*}
The boundary conditions in \eqref{eq: BC 2} become
\begin{align*}
\frac{\beta}{\Delta x} \widetilde{w}_z(1,t) - \left(\widetilde{\alpha}_{1}(1,t)-u(t)- {\dot{\overline{x}}(t)}\right)\widetilde{w}(1,t)
  =  {\overline{\sigma}(t)},~&t\in\mathbb{R}_{\geq 0}, \\
  \frac{\beta}{\Delta x} \widetilde{w}_z(0,t) - \left(\widetilde{\alpha}_{1}(0,t)-u(t)- {\dot{\underline{x}}(t)}\right)\widetilde{w}(0,t)
  =  {\underline{\sigma}(t)},~&t\in\mathbb{R}_{\geq 0},\\
 \frac{\beta}{\Delta x} \widetilde{v}_z(1,t) - \left(\widetilde{\alpha}_{0}(1,t)-u(t)-\dot{\overline{x}}(t)\right)\widetilde{v}(1,t)
    = {-\overline{\sigma}(t)}, ~&t\in\mathbb{R}_{\geq 0}, \\
     \frac{\beta}{\Delta x} \widetilde{v}_z(0,t) - \left(\widetilde{\alpha}_{0}(0,t)-u(t)-\dot{\underline{x}}(t)\right)\widetilde{v}(0,t)
    = {-\underline{\sigma}(t)}, ~&t\in\mathbb{R}_{\geq 0}.
\end{align*}

Let $\widehat{\beta}= \frac{\beta}{(\Delta x)^2}$, $\widehat{u}= \frac{1}{\Delta x}u$, $\widehat{\alpha}_j= \frac{1}{\Delta x}\widetilde{\alpha}_j$, {$\widehat{\overline{\sigma}}= \frac{1}{\Delta x} \overline{\sigma},\widehat{\underline{\sigma}}= \frac{1}{\Delta x} \underline{\sigma}$} and $w^0(x)=w^0(z\Delta x+\underline{x}):=\widetilde{w}^0(z),v^0(x)=v^0(z\Delta x+\underline{x}):=\widetilde{v}^0(z)$. Noting that {$u(t) = \dot{\underline{x}}(t) = \dot{\overline{x}}(t)$}, we obtain the following system:
\begin{subequations}\label{eq: foced FPE'}
\begin{align}
  \widetilde{w}_t = \left(\widehat{\beta}\widetilde{w}_z- \left(\widehat{\alpha}_{1}-2{\widehat{u}}\right)\widetilde{w}\right)_z
                                         +\widetilde{\delta}, \ &(z,t)\in Q_\infty, \label{eq: foced FPE ON'}\\
  \widetilde{v}_t = \left(\widehat{\beta}\widetilde{v}_z- \left(\widehat{\alpha}_{0}-2{\widehat{u}}\right)\widetilde{v}\right)_z
                                         -\widetilde{\delta}, \ &(z,t)\in Q_\infty, \label{eq: foced FPE OFF'}\\
  \widehat{\beta} \widetilde{w}_z(1,t) - \left(\widehat{\alpha}_{1}(1,t)- 2\widehat{u}(t)\right)\widetilde{w}(1,t)
        =   {\widehat{\overline{\sigma}}(t)}, \ &t\in\mathbb{R}_{\geq 0}, \label{B1}\\
  \widehat{\beta} \widetilde{w}_z(0,t) - \left(\widehat{\alpha}_{1}(0,t)- 2\widehat{u}(t)\right)\widetilde{w}(0,t)
        = {\widehat{\underline{\sigma}}(t)}, \ &t\in\mathbb{R}_{\geq 0}, \label{B2}\\
  \widehat{\beta} \widetilde{v}_z(1,t) - \left(\widehat{\alpha}_{0}(1,t)- 2\widehat{u}(t)\right)\widetilde{v}(1,t)
       =  {-\widehat{\overline{\sigma}}(t)}, \ &t\in\mathbb{R}_{\geq 0},\label{B3}\\
  \widehat{\beta} \widetilde{v}_z(0,t) - \left(\widehat{\alpha}_{0}(0,t)- 2\widehat{u}(t)\right)\widetilde{v}(0,t)
       =  {-\widehat{\underline{\sigma}}(t)}, \ &t\in\mathbb{R}_{\geq 0},\label{B4}\\
  \widetilde{w}(z,0)= \widetilde{w}^0(z),\widetilde{v}(z,0)= \widetilde{v}^0(z), \ &z\in (0,1).\label{I1}
\end{align}
\end{subequations}
It should be noted that transforming the original system~\eqref{eq: foced FPE} into \eqref{eq: foced FPE'} is only for the purpose of simplifying control design and analysis. Whereas, control implementation and operation will be performed in the original frame with moving temperature bounds.

\section{Aggregate power tracking control design}\label{Sec: Control design}
In the considered power load tracking control of the whole TCL population, the weighted total power consumption is chosen as the system output
\begin{align}\label{eq: output}
  y(t) =& \dfrac{P}{\eta}\int_{\underline{x}(t)}^{\overline{x}(t)} \left(a x + b(t)\right)w(x,t)\diff x,
\end{align}
where $\eta$ is the load efficiency, $a$ is a non-zero constant, and {$b(t)$ is a $C^2$-function}. Note that the weighting function $ax+b$ {with $a \neq 0$ introduced in system output defined above is to guarantee that the input-output dynamics of the system are well-defined in terms of characteristic index \cite{Christofides1996,Christofides2001}}, which is a generalization of the concept of relative degree for finite dimensional systems \cite{Isidori:1995,Khalil:2002,Levine:2009}. Moreover, although other forms of weighting function can also be considered, the one given in \eqref{eq: output} is a convenient choice that facilitates control design and well-posedness and stability analysis, and is meaningful for the application considered in this work as illustrated later in Section~\ref{Sec: Simulation Study}. However, as the boundaries in the computation of the output defined in \eqref{eq: output} are time-varying, it is not suitable for control design and analysis. For this reason, we transform the output into a form with fixed boundaries in the normalized coordinates. Specifically, letting $\widehat{a} = a{(\Delta x)}^2$ and $\widehat{b} = \Delta x(a\underline{x}+b)$, the output can be expressed in the normalized coordinates as
\begin{align}\label{eq: output'}
  y(t)  =\dfrac{P}{\eta}\int_{0}^{1} \left(\widehat{a}z+\widehat{b}\right)\widetilde{w}(z,t)\diff z.
\end{align}
Moreover, we have $\dot{\widehat{b}} = \Delta x (a\dot{\underline{x}}+\dot{b}) = \widehat{a}\widehat{u}+\Delta x\dot{b}$.

Let $y_d(t)$ be the desired aggregate power, which is usually a sufficiently smooth function, and denote by $e(t) = y(t) - y_d(t)$ the tracking error. We can derive from \eqref{eq: foced FPE'} and \eqref{eq: output'} the tracking error dynamics
\begin{align*}
  \dfrac{\diff{e}}{\diff{t}}
  =& \dfrac{P}{\eta}\int_{0}^{1}\left(\left(\widehat{a}z + \widehat{b}\right)\widetilde{w}_t+ \dot{\widehat{b}} \widetilde{w}\right)\diff z- \dot{y}_d(t)\\
  =& \dfrac{P}{\eta}\int_{0}^{1}\left(\widehat{a}z+\widehat{b}\right)\left(\widehat{\beta} \widetilde{w}_{z}- (\widehat{\alpha}_{1}
     -2\widehat{u})\widetilde{w} \right)_z\diff z  + \dfrac{P}{\eta}\int_{0}^{1}  {\left({\widehat{a}} \widehat{u}+ \Delta x\dot{b}\right)} \widetilde{w} \diff z
 + \dfrac{P}{\eta}\int_{0}^{1} \left(\widehat{a}z+\widehat{b}\right)\widetilde{\delta} \diff z - \dot{y}_d(t).
\end{align*}
Performing integration by parts on the first term on the right-hand side of the above equation while applying the boundary conditions given in \eqref{eq: foced FPE'}, we get
\begin{align*}
  \dfrac{\diff{e}}{\diff{t}}
     =&-\dfrac{P}{\eta}\int_{0}^{1}\widehat{a}\left(\widehat{\beta} \widetilde{w}_{z}- (\widehat{\alpha}_{1} -{\widehat{u}})\widetilde{w} \right)\diff z+ {\dfrac{P}{\eta}\int_{0}^{1} \Delta x\dot{b}\widetilde{w}\diff z}+ \dfrac{P}{\eta}\left[{\left(\widehat{a}+\widehat{b}\right)\widehat{\overline{\sigma}}(t) {-\widehat{b}\widehat{\underline{\sigma}}(t)}}
            + \int_{0}^{1}\left(\widehat{a}z+\widehat{b}\right)\widetilde{\delta}\diff z\right]- \dot{y}_d(t).
\end{align*}
Note that $ \int_{0}^{1} \widetilde{w}_{z}\diff z=\widetilde{w}(1,t)-\widetilde{w}(0,t)$ and let
\begin{align}
 \widehat{u}(t)
      = & -\frac{\widehat{\beta}(\widetilde{w}(1,t)-\widetilde{w}(0,t))-\displaystyle\int_{0}^{1}\left(\widehat{\alpha}_{1}+ { \Delta x\dot{b}}\right)\widetilde{w}
            \diff{z} + \frac{\eta}{\widehat{a}P}\phi(t)}{\displaystyle {\int_{0}^{1}\widetilde{w}\diff z}}, \label{eq: control u}\\
 \Gamma(t)=&  \dfrac{P}{\eta}\left[{\left(\widehat{a}+\widehat{b}\right)\widehat{\overline{\sigma}}(t)-\widehat{b}\widehat{\underline{\sigma}}(t)}                      + \int_{0}^{1}\left(\widehat{a}z+\widehat{b}\right)\widetilde{\delta}\diff z\right], \label{eq: control Gamma}
\end{align}
where $\phi(t)$ is an auxiliary control.
Then the tracking error dynamics become
\begin{equation}\label{eq: error dynamics 1}
  \dfrac{\diff{e}}{\diff{t}}(t) = \phi(t) - \dot{y}_d(t) + \Gamma(t), \ \ e(0) = e_0,
\end{equation}
where $e_0$ is the initial regulation error.

In the original coordinates, the control is expressed as
\begin{align}
   u(t)   =&-\frac{\beta(w(\overline{x},t)-w(\underline{x},t))-\displaystyle\int_{\underline{x}}^{\overline{x}}{\left(\alpha_{1}+ \dot{b}\right)}w
            \diff{x} + \frac{\eta}{aP}\phi(t)}{\displaystyle   {\Delta x }\int_{\underline{x}}^{\overline{x}}w\diff{x}},\label{eq: control u'}\\
   \Gamma(t)=&\dfrac{P}{\eta}\left[{\left(a\overline{x}+b\right)\overline{\sigma}
   (t)-\left(a\underline{x}+b\right)\underline{\sigma}(t)}
                            +\int_{\underline{x}}^{\overline{x}}\left(ax+b\right)\delta \diff x\right].  \label{eq: control Gamma'}
\end{align}

Letting
\begin{equation}\label{eq: stabilizer}
  \phi(t) = \dot{y}_d(t) -k_0e(t),
\end{equation}
where $k_0 > 0$ is the controller gain, we obtain
\begin{equation*}
  \dfrac{\diff{e}}{\diff{t}}(t) = -k_0 e(t) + \Gamma(t), \ \ e(0) = e_0.
\end{equation*}

It can be seen from \eqref{eq: control u'} that as the deadband is not empty, the characteristic index of the input-output dynamics of the system is $1$ if $\int_{\underline{x}}^{\overline{x}}w\diff{x}$, representing the accumulation of the loads in ON-state, is not null. This condition can be easily fulfilled for a large enough TCL population for which it is reasonable that at least one TCL is in the whole population in ON-state all the time. In addition, due to the fact that at any moment only a small portion of TCLs may change their state between ON and OFF and the {boundedness of $b$} can be guaranteed by an appropriate design, $\Gamma(t)$ is uniformly bounded, i.e., $|\Gamma(t)| \leq \Gamma_{\infty}$ for all $t > 0$ with $\Gamma_{\infty}$ a positive constant. Thus, the control given in \eqref{eq: stabilizer} guarantees that the trajectory of the system~\eqref{eq: error dynamics 1} is globally uniformly bounded, as stated later in Theorem~\ref{theorem 12}. Moreover, the amplitude of regulation error $e(t)$ may be reduced by increasing the control gain $k_0$. {In addition, it should be noted that the proposed control scheme is robust with respect to $\delta(x,t)$, $\overline{\sigma}(t)$, and $\underline{\sigma}(t)$, which are treated as disturbances. Finally, as there is no need to compute $\Gamma(t)$,} the TCLs do not need to signal the instantaneous switching operations. This will allow greatly simplifying the implementation and making the control scheme insensitive to timing constraints.

\section{Well-posedness and stability of the closed-loop system}\label{Sec: Well-posedness Stability}
This section is dedicated to addressing two essential properties of the closed-loop system, namely the well-posedness of the setting in terms of the existence and the uniqueness of the solution and the closed-loop stability.
\subsection{Well-posedness assessment}\label{Sec: Well-posedness}
It is well-known that in the study of classical PDEs (e.g., Laplace equation, heat equations, and wave equations), usually some very specific types of boundary conditions, e.g., Dirichlet, Neumann and Robin boundary conditions, are associated with these equations. It is often physically obscure whether a boundary condition is appropriate or not for a given PDE. Therefore, this aspect has to be clarified by a fundamental mathematical insight.

According to Hadamard's principe \cite{Hadamard:1923}, an initial-boundary value problem (IBVP) of PDEs is said to be ``well-posed'' if:
\begin{itemize}
  \item [(i)] it has a solution on a prescribed domain for all suitable boundary data;
 \item [(ii)] the solution is uniquely determined by such data;
 \item [(iii)] the solution is also continuously determined by such data.
\end{itemize}
In PDE theory, the study of well-posedness is mainly focused on the existence and the uniqueness of the solutions. From an application point of view, the importance of the well-posedness is that it assures the validity of a model in the sense that its mathematical formulation complies with the nature of the real-world system. However, unlike finite-dimensional systems described by ordinary differential equations for which the existence of unique solutions of a wide variety of systems can be confirmed under some regular conditions (see, e.g., \cite[Ch.~3]{Khalil:2002}), there do not exist generic results for the well-posedness of PDEs. Therefore, well-posedness assessment for PDEs is usually carried out on a case-by-case basis. This motivates our study presented in this section.

%

We establish first the well-posedness of the closed-loop system composed of \eqref{eq: foced FPE}, \eqref{eq: BC 2}, \eqref{initialwv}, \eqref{eq: output}, \eqref{eq: error dynamics 1}, \eqref{eq: control u'}, \eqref{eq: control Gamma'}, and \eqref{eq: stabilizer}, {or equivalently}, \eqref{eq: foced FPE'}, \eqref{eq: output'}, \eqref{eq: control u}, \eqref{eq: control Gamma}, \eqref{eq: error dynamics 1}, and \eqref{eq: stabilizer}. For this purpose, in the present work, we always impose the following basic assumptions:
\begin{enumerate}
  \item[\textbf{(A1)}] ${\beta},P,\eta,k_0 \in\mathbb{R}_+$, $a\in \mathbb{R}\setminus\{0\}$;
  \item[\textbf{(A2)}] ${\overline{\sigma},\underline{\sigma}\in  L^1_{\text{loc}}(\mathbb{R}_{\geq 0})}$,  $b\in  C^2(\mathbb{R}_{\geq 0})$,
      $y_d\in C^2(\mathbb{R}_{\geq 0})$,
   $\overline{x},\underline{x}\in C^1(\mathbb{R}_{\geq 0})$,  $\Delta x:=\overline{x}-\underline{x}$ is a positive constant;
  \item[\textbf{(A3)}] for any {$T\in \mathbb{R}_+$ and $t\in[0,T]$, $ \delta \in L^1(\Omega_T)$}, ${\alpha}_1,{\alpha}_0 \in C^1([\underline{x}(t),\overline{x}(t)])$, $w^0,v^0\in \mathcal {H}^{2+\theta}([\underline{x}(t), \overline{x}(t)])$ with a constant $\theta\in (0,1)$;
  \item[\textbf{(A4)}] for any $t\in\mathbb{R}_{\geq 0}$,  $w^0,v^0,\delta,\overline{\sigma}$ and $\underline{\sigma}$ satisfy the following conditions:
 \begin{subequations}
\begin{align}
&{\mathscr{B}_1[w^0](\overline{x},0)=\overline{\sigma}(0),
\mathscr{B}_1[w^0](\underline{x},0)= \underline{\sigma}(0),}\label{09021}\\
&{\mathscr{B}_0[v^0](\overline{x},0)=-\overline{\sigma}(0),\mathscr{B}_0[v^0](\underline{x},0)=-\underline{\sigma}(0),}\label{09022}\\
  &\bigg|\int_{\underline{x}(0)}^{\overline{x}(0)}w^0(x)\diff x\bigg|>0, \label{eq: ON on 0}\\
  &\bigg|{\displaystyle\int_{\underline{x}(0)}^{\overline{x}(0)}w^0(x)\diff x +\displaystyle\int_{0}^t\int_{\underline{x}(s)}^{\overline{x}(s)} \delta(x,s)\diff x\diff s+ \int_{0}^t\left(\overline{\sigma} (s)-\underline{\sigma}(s)\right) \diff s}\bigg|\geq \frac{\delta_0}{2},\label{+28}
\end{align}
\end{subequations}
where $\delta_0$ is positive constant.
\end{enumerate}
Note that \eqref{09021} and \eqref{09022} are compatibility conditions for the well-posedness and the conditions given in \eqref{eq: ON on 0} and \eqref{+28} mean that at any time the TCLs are not all in OFF-state. {Moreover, the assumptions on the discontinuities and local integrabilities of  $\delta$, $\overline{\sigma}$, and $\underline{\sigma}$ reflect well the nature of the considered problem.}  {However, since $\delta$, $\overline{\sigma}$, and $\underline{\sigma}$ are discontinuous, the closed-loop system composed of  \eqref{eq: foced FPE}, \eqref{eq: BC 2}, \eqref{initialwv}, \eqref{eq: output}, \eqref{eq: error dynamics 1}, \eqref{eq: control u'}, \eqref{eq: control Gamma'}, and \eqref{eq: stabilizer} (or \eqref{eq: foced FPE'}, \eqref{eq: output'}, \eqref{eq: control u}, \eqref{eq: control Gamma}, \eqref{eq: error dynamics 1}, and \eqref{eq: stabilizer}) does not admit any classical solution that satisfies the equations pointwisely. To address the well-posedness, we consider certain smooth solutions that satisfy the equations in the sense of distribution. It is worth noting that most of the above assumptions are required for assuring the existence of such smooth solutions to the considered problem, which can eventually be relaxed if we consider solutions in much {weaker} senses. }
%
\begin{definition}
Let $D$ be {an open or closed} domain in $\mathbb{R}^i (i=1,2)$. For two locally integrable functions $f$ and $g$ defined on $D$, we say that $f=g$ in the sense of distribution, if $\int_Df(s) \varphi(s)\diff s=\int_Dg(s) \varphi(s)\diff s$ holds for any $\varphi\in C_0^\infty(D)$.
\end{definition}
\begin{definition}
\begin{enumerate}
\item[(i)]
We say that $(w,v)\in C^{2,1}(\overline{\Omega}_\infty)\times C^{2,1}(\overline{\Omega}_\infty)$ is a distributional solution of  the closed-loop system composed of \eqref{eq: foced FPE}, \eqref{eq: BC 2}, \eqref{initialwv}, \eqref{eq: output}, \eqref{eq: error dynamics 1}, \eqref{eq: control u'}, \eqref{eq: control Gamma'}, and \eqref{eq: stabilizer}, if such $(w,v)$ satisfies \eqref{eq: foced FPE}, {\eqref{eq: BC 2}, and \eqref{initialwv} in the sense of distribution}.
\item[(ii)] We say that $(w,v)\in C^{2,1}(\overline{Q}_{\infty})\times C^{2,1}(\overline{Q}_{\infty})$ is a  distributional solution of the closed-loop system composed of \eqref{eq: foced FPE'}, \eqref{eq: output'}, \eqref{eq: control u}, \eqref{eq: control Gamma}, \eqref{eq: error dynamics 1}, and \eqref{eq: stabilizer}, if such  $(\widetilde{w},\widetilde{v})$ satisfies \eqref{eq: foced FPE ON'}, \eqref{eq: foced FPE OFF'}, \eqref{B1}, \eqref{B2}, \eqref{B3}, \eqref{B4}, and \eqref{I1} in the sense of distribution.
    \end{enumerate}
\end{definition}
Essentially, the existence of a  solution to an IBVP can be established via regularity analysis and \textit{a priori} estimates of the solution. The result on the existence of a solution to the considered problem is given in the following theorem.

\begin{theorem}\label{well-posedness'}The closed-loop system composed of \eqref{eq: foced FPE}, \eqref{eq: BC 2}, \eqref{initialwv}, \eqref{eq: output}, \eqref{eq: error dynamics 1}, \eqref{eq: control u'}, \eqref{eq: control Gamma'}, and \eqref{eq: stabilizer} admits a distributional solution $(w,v)\in C^{2,1}(\overline{\Omega}_\infty)\times C^{2,1}(\overline{\Omega}_\infty)$.
\end{theorem}
By virtue of the variable transformation: $z=\frac{x-\underline{x}}{\Delta x}$, Theorem \ref{well-posedness} is a direct result of the following proposition, whose proof is provided in Appendix~\ref{Appendix: well-posedness proof}.
\begin{proposition}\label{well-posedness}
The closed-loop system composed of \eqref{eq: foced FPE'}, \eqref{eq: output'}, \eqref{eq: control u}, \eqref{eq: control Gamma}, \eqref{eq: error dynamics 1}, and \eqref{eq: stabilizer} admits a distributional solution $(\widetilde{w},\widetilde{v})\in C^{2,1}(\overline{Q}_{\infty})\times C^{2,1}(\overline{Q}_{\infty})$.
\end{proposition}



{It is worth noting that if $\delta$, $\overline{\sigma}$ and $\underline{\sigma}$ are continuous, such distributional solutions of the closed-loop system composed of  \eqref{eq: foced FPE}, \eqref{eq: BC 2}, \eqref{initialwv}, \eqref{eq: output}, \eqref{eq: error dynamics 1}, \eqref{eq: control u'}, \eqref{eq: control Gamma'}, and \eqref{eq: stabilizer} (or \eqref{eq: foced FPE'}, \eqref{eq: output'}, \eqref{eq: control u}, \eqref{eq: control Gamma}, \eqref{eq: error dynamics 1}, and \eqref{eq: stabilizer})  become classical solutions that satisfy the equations pointwisely.}

Regarding the uniqueness of the distributional solution of the closed-loop system, we have {the following theorem, whose proof is provided in Appendix \ref{Appendix: uniqueness}.}
\begin{theorem}\label{uniqueness}
Let $(w_1,v_1),(w_2,v_2)\in C^{2,1}(\overline{\Omega}_\infty)\times C^{2,1}(\overline{\Omega}_\infty)$ be two distributional solutions of the closed-loop system composed of \eqref{eq: foced FPE}, \eqref{eq: BC 2}, \eqref{initialwv}, \eqref{eq: output}, \eqref{eq: error dynamics 1}, \eqref{eq: control u'}, \eqref{eq: control Gamma'}, and \eqref{eq: stabilizer}. If for any $t\in \mathbb{R}_{\geq 0}$,
 $w_1(\overline{x}(t),\cdot)-w_1(\underline{x}(t),\cdot)={w}_2(\overline{x}(t),\cdot)-w_2(\underline{x}(t),\cdot)$ in $\mathbb{R}_{\geq 0}$, then  $(w_1,v_1)=(w_2,v_2)$ in $\overline{\Omega}_\infty$.
\end{theorem}
\subsection{Stability analysis}\label{Sec: Stability Analysis}
We assess first the stability of the error dynamics~\eqref{eq: error dynamics 1} with the control given in \eqref{eq: stabilizer}. As the regulation error, $e(t)$, is governed by a first-order finite dimensional linear system, it is straightforward to have the following result.
\begin{theorem}\label{theorem 12}
The regulation error $e(t)$ is determined by
\begin{equation*}
  e(t)=e(0)\e^{-k_0t}+\int_{0}^t\Gamma(t)\e^{-k_0(t-s)}\diff s.\label{1125+4}
\end{equation*}
Furthermore, if $|\Gamma(t)| \leq \Gamma_{\infty}$ with a positive constant $\Gamma_{\infty}$ for all $t > 0$, then
\begin{equation*}
|e(t)|\leq |e(0)|\e^{-k_0t}+\frac{\Gamma_{\infty}}{k_0}\left(1-\e^{-k_0t}\right),\forall t>0.
\end{equation*}
\end{theorem}
To assure the closed-loop stability, we need to prove that the internal dynamics composed of \eqref{eq: foced FPE}, \eqref{eq: output} and \eqref{eq: control u'} are stable provided Theorem~\ref{theorem 12} holds true. Towards this aim, we establish some essential properties of the closed-loop solution stated in the following 2~propositions.
\begin{proposition}\label{prop. 8} Let $(w,v)\in C^{2,1}(\overline{\Omega}_\infty)\times C^{2,1}(\overline{\Omega}_\infty)$ be a distributional solution of the closed-loop system composed of \eqref{eq: foced FPE}, \eqref{eq: BC 2}, \eqref{initialwv}, \eqref{eq: output}, \eqref{eq: error dynamics 1}, \eqref{eq: control u'}, \eqref{eq: control Gamma'}, and \eqref{eq: stabilizer}. For any $t\in \mathbb{R}_{\geq 0}$, it holds:
\newline
\text{(i)}  $\displaystyle\int_{\underline{x}(t)}^{\overline{x}(t)}w(x,t)\diff x
    =  \displaystyle\int_{\underline{x}(0)}^{\overline{x}(0)}w^0(x)\diff x +\displaystyle\int_{0}^t\int_{\underline{x}(s)}^{\overline{x}(s)} \delta(x,s)\diff x\diff s+{ \int_{0}^t\left(\overline{\sigma} (s)-\underline{\sigma}(s)\right) \diff s }$;
\newline
\text{(ii)} $\displaystyle\int_{\underline{x}(t)}^{\overline{x}(t)}v(x,t)\diff x
    =\displaystyle\int_{\underline{x}(0)}^{\overline{x}(0)}v^0(x)\diff x -\displaystyle\int_{0}^t\int_{\underline{x}(s)}^{\overline{x}(s)} \delta(x,s)\diff x\diff s- { \int_{0}^t\left(\overline{\sigma} (s)-\underline{\sigma}(s)\right) \diff s }$.\\
Therefore, \eqref{agg} holds true.
\end{proposition}

\begin{proof}
Let $\widetilde{w}$ satisfy \eqref{eq: foced FPE'}. We have then
\begin{align*}
 \frac{\text{d} }{\text{d}t}\int_{0}^{1} \widetilde{w}(z,t)\diff z
 &=\int_{0}^{1} \widetilde{w}_t(z,t)\diff z\notag\\
 &=\int_{0}^{1}\left(\widehat{\beta}\widetilde{w}_z- (\widehat{\alpha}_{1}-2u)\widetilde{w}\right)_z\diff z+\int_{0}^{1} \widetilde{\delta}(z,t)\diff z\notag\\
 &= \widehat{\overline{\sigma}} (t)-\widehat{\underline{\sigma}} (t)+\int_{0}^{1} \widetilde{\delta}(z,t)\diff z.
\end{align*}
It follows that
\begin{align*}
 \int_{0}^{1} \widetilde{w}(z,t)\diff z = \int_{0}^{1} \widetilde{w}^0(z)\diff z + \int_{0}^t\int_{0}^{1} \widetilde{\delta}(z,s)\diff z\diff s+  \int_{0}^t\left(\widehat{\overline{\sigma}} (s)-\widehat{\underline{\sigma}} (s)\right)\diff s.
\end{align*}
Using the transformation $z=\frac{x-\underline{x}}{\Delta x}$, we obtain the result claimed in~(i). The result claimed in~(ii) can be obtained in a similar way.
$\hfill\Box$
\end{proof}

\begin{proposition}\label{prop. 7}
Let $(w,v)\in C^{2,1}(\overline{\Omega}_\infty)\times C^{2,1}(\overline{\Omega}_\infty)$ be a distributional solution of the closed-loop system composed of \eqref{eq: foced FPE}, \eqref{eq: BC 2}, \eqref{initialwv}, \eqref{eq: output}, \eqref{eq: error dynamics 1}, \eqref{eq: control u'}, \eqref{eq: control Gamma'}, and \eqref{eq: stabilizer}. For any $t\in \mathbb{R}_{\geq 0}$, it holds:
\begin{align*}
(i)\quad\|w(\cdot,t)\|_{L^1(\underline{x}(t),\overline{x}(t))}\leq & \|w^0\|_{L^1(\underline{x}(0),\overline{x}(0))}+\displaystyle\int_0^t\int_{\underline{x}(s)}^{\overline{x}(s)}{\delta} (x,s)\text{sgn}(w(x,s)) \diff x\diff s{+\displaystyle\int_0^t \overline{\sigma} (s)\text{sgn}(w(\overline{x}(s),s))\diff s} \notag\\
&
{-\displaystyle\int_0^t \underline{\sigma}(s)\text{sgn}(w(\underline{x}(s),s))\diff s};\\
(ii)\quad\|v(\cdot,t)\|_{L^1(\underline{x}(t),\overline{x}(t))}\leq & \|v^0\|_{L^1(\underline{x}(0),\overline{x}(0))}-\displaystyle\int_0^t\int_{\underline{x}(s)}^{\overline{x}(s)}{\delta} (x,s)\text{sgn}(v(x,s)) \diff x\diff s{-\displaystyle\int_0^t \overline{\sigma} (s)\text{sgn}(v(\overline{x}(s),s))\diff s} \notag\\
&{+\displaystyle\int_0^t \underline{\sigma}(s)\text{sgn}(v(\underline{x}(s),s))\diff s},
\end{align*}
where $ \text{sgn}(f):=1 $ if $f>0$; $ \text{sgn}(f):=-1 $ if $f<0$; $ \text{sgn}(f):=0 $ if $f=0$.
\end{proposition}
\begin{proof}
By virtue of the variable transformation $z=\frac{x-\underline{x}}{\Delta x}$, it suffices to let $n \rightarrow \infty$ in \eqref{-38} ( see Appendix \ref{Appendix: well-posedness proof}), which leads to the desired result.$\hfill\Box$
\end{proof}
It should be mentioned that Proposition~\ref{prop. 7} can still not guarantee the stability of the internal dynamics. In fact by checking Claim~(i), it is obvious that if $w$ has the same sign as $\delta$ over $\Omega_T$,
%
%
then $\int_0^t\int_{\underline{x}}^{\overline{x}}\delta(x,s) \text{sgn}(w(x,s)) \diff x\diff s =\int_0^t\int_{\underline{x}}^{\overline{x}}|\delta(x,s)|\diff x\diff s$. However, it is nature to do not impose any assumption on the global $L^1$-integrability of $\delta$. That is to say, $\int_0^t\int_{\underline{x}}^{\overline{x}}|\delta(x,s)|\diff x\diff s$ may tend to $\infty$ as $t\rightarrow \infty$. Similarly, $\int_0^t |\overline{\sigma} (s)|\diff s$ and $\int_0^t |\underline{\sigma} (s)|\diff s$ may also tend to $\infty$ as $t\rightarrow \infty$.
 Consequently, the boundedness of $\|w(\cdot,t)\|_{L^1(\underline{x}(t),\overline{x}(t))}$ cannot yet be guaranteed if $t\rightarrow \infty$. The same argument also holds true for Claim~(ii).

To establish the closed-loop stability of the internal dynamics, we note first that due to the fact that the considered problem is conservative in terms of the total number of the TCLs in the population, there must be positive constants $M,M'$ such that $\left|\iint_\Omega\delta(x,{t}) \diff x\diff t\right|<M$ for any Lebesgue measurable set $\Omega\subset \Omega_{\infty}$, {and $\left| \int_I\overline{\sigma} (s)\diff s\right|<M',\left|\int_I\underline{\sigma} (s)\diff s\right|<M'$} for any Lebesgue measurable set $I\subset [0,+\infty)$. We have then the following result.

\begin{theorem}\label{stability}
Assume that there exist positive constants $M,M'$ such that $\left|\iint_\Omega\delta(x,{t}) \diff x\diff t\right|<M$ for any Lebesgue measurable set $\Omega\subset \Omega_{\infty}$, and {$\left| \int_I\overline{\sigma} (s)\diff s\right|<M',\left|\int_I\underline{\sigma} (s)\diff s\right|<M'$} for any Lebesgue measurable set $I\subset [0,+\infty)$.
Let $(w,v)\in C^{2,1}(\overline{\Omega}_\infty)\times C^{2,1}(\overline{\Omega}_\infty)$ be a distributional solution of the closed-loop system composed of \eqref{eq: foced FPE}, \eqref{eq: BC 2}, \eqref{initialwv}, \eqref{eq: output}, \eqref{eq: error dynamics 1}, \eqref{eq: control u'}, \eqref{eq: control Gamma'}, and \eqref{eq: stabilizer}. For any $t\in \mathbb{R}_{\geq 0}$, it holds:
\newline
(i)\quad$
\|w(\cdot,t)\|_{L^1(\underline{x}(t),\overline{x}(t))}\leq  \|w^0\|_{L^1(\underline{x}(0),\overline{x}(0))}
+2M+2M'<+\infty;$
\newline
(ii)\quad$\|v(\cdot,t)\|_{L^1(\underline{x}(t),\overline{x}(t))}\leq  \|v^0\|_{L^1(\underline{x}(0),\overline{x}(0))}+2M+2M'<+\infty.$
\end{theorem}
\begin{proof}
For any $T\in \mathbb{R}_{+}$, we consider the closed-loop system composed of \eqref{eq: foced FPE'}, \eqref{eq: output'}, \eqref{eq: control u}, \eqref{eq: control Gamma}, \eqref{eq: error dynamics 1}, and \eqref{eq: stabilizer} over the domain $Q_T$.
For $t=0$,  we have immediately
\begin{align*}
\|w(\cdot,0)\|_{L^1(\underline{x}(0),\overline{x}(0))}= \|w^0\|_{L^1(\underline{x}(0),\overline{x}(0))}\leq \|w^0\|_1
+2{M}+2M'<+\infty.
\end{align*}
For any $t\in (0,T]$, let $Q^+=\{(z,s)\in \overline{Q}_{t};\widetilde{w}(z,s)>0\} $, $Q^-=\{(z,s)\in \overline{Q}_{t};\widetilde{w}(z,s)<0\} $, $\Omega^+=\{(x,s)\in \overline{\Omega}_{t};w(x,s)>0\} $ and $\Omega^-=\{(x,s)\in \overline{\Omega}_{t};w(x,s)<0\} $, $J^+=\{s\in [0,t];\widetilde{w}(1,s)>0\} $, $J^-=\{s\in [0,t];\widetilde{w}(1,s)<0\} $, $I^+=\{s\in [0,t];w(\overline{x}(s),s)>0\} $ and $I^-=\{s\in [0,t];w(\underline{x}(s),s)<0\} $. By Proposition~\ref{prop. 7} and the variable transformation $z=\frac{x-\underline{x}}{\Delta x}$, we have
\begin{align*}
&\|\widetilde{w}(\cdot,t)\|_{L^1(\underline{x}(t),\overline{x}(t))}\\
\leq&  \|\widetilde{w}^0\|_{L^1(\underline{x}(0),\overline{x}(0))} {+\displaystyle\int_0^t \widehat{\overline{\sigma} } (s)\text{sgn}(\widetilde{w}(1,s))\diff s} {-\displaystyle\int_0^t \widehat{\underline{\sigma}}(s)\text{sgn}(\widetilde{w}(0,s))\diff s}
+\displaystyle\int_0^t\int_{0}^{1}\widetilde{\delta}(z,s) \text{sgn}(\widetilde{w}) \diff z\diff s\\
= &\|\widetilde{w}^0\|_{L^1(\underline{x}(0),\overline{x}(0))} {+\displaystyle\int_{J^+} \widehat{\overline{\sigma} } (s)\diff s} {+\displaystyle\int_{J^-} \widehat{\underline{\sigma}}(s)\diff s}+\iint_{Q^+}\widetilde{\delta}(z,s) \diff z\diff s -\iint_{Q^-}\widetilde{\delta}(z,s)\diff z\diff s\\
\leq &\|\widetilde{w}^0\|_{L^1(\underline{x}(0),\overline{x}(0))} {+\left|\displaystyle\int_{J^+} \widehat{\overline{\sigma} } (s)\diff s\right|} {+\left|\displaystyle\int_{J^-} \widehat{\underline{\sigma}}(s)\diff s\right|}
+\left|\iint_{Q^+}\widetilde{\delta}(z,s)\diff z\diff s\right| +\left|\iint_{Q^-}\widetilde{\delta}(z,s)\diff z\diff s\right|\\
    = &\|\widetilde{w}^0\|_{L^1(\underline{x}(0),\overline{x}(0))} {+\frac{1}{\Delta x}\left|\displaystyle\int_{I^+} {\overline{\sigma} } (s)\diff s\right|} {+\frac{1}{\Delta x}\left|\displaystyle\int_{I^-} {\underline{\sigma}}(s)\diff s\right|}
+\frac{1}{\Delta x}\left|\iint_{\Omega^+}{\delta}(x,s) \diff x\diff s\right|
+\frac{1}{\Delta x}\left|\iint_{\Omega^-}{\delta}(x,s)\diff x\diff s\right|\\
\leq & \|\widetilde{w}^0\|_{L^1(\underline{x}(0),\overline{x}(0))}+\frac{2M'}{\Delta x}+\frac{2M}{\Delta x}
\notag\\<&{+\infty},
\end{align*}
which, along with the variable transformation $z=\frac{x-\underline{x}}{\Delta x}$, yields Claim~(i). Claim~(ii) can be proved in the same way.
$\hfill\Box$
\end{proof}
\section{Simulation study}\label{Sec: Simulation Study}
We illustrate the developed control scheme through a benchmark problem corresponding to a population of air-conditioners (see, e.g., \cite{Callaway:2009,Ghanavati:2018,Moura:2014,moura2013modeling,ZLZL:2019}). The simulated system contains 10,000~TCLs. The parameters of the system and the controllers are listed in Table~\ref{Tab: system parameters}. As proposed in \cite{Callaway:2009,Ghanavati:2018,Moura:2014,moura2013modeling}, the heterogeneity of this population is generated by the variation of the thermal capacitance of TCLs, which is supposed to follow a log-normal distribution in this study resulting in a difference in term of the thermal constant ($RC$) of the TCLs up to 3.5~times. Moreover, we considered in this work a case where the ambient temperature varies over time. As the measurement of the ambient temperature is available for control implementation, the actual form of the ambient temperature can be arbitrary. We used then the one shown in Fig.~\ref{Fig: Tracking_ext} in the simulation study.
\begin{table}[thpb]
\centering{{Table 1. System parameters}\label{Tab: system parameters}}\\
\begin{tabular}{|l|c|c|c|c|c|}
\hline
Parameters & Description [Unit] & Value\\
\hline
$R$ & Thermal resistance [$^\circ$C/kW] & 2\\
$C$ & Thermal capacitance [kWh/$^\circ$C] & $\sim\mathcal{N}$(10,~3)\\
$P$ & Thermal power [kW] & 14\\
$x_{ie}$ & Ambient temperature [$^\circ$C] & [28, 32]\\
$\eta$ & Load efficiency & 2.5 \\
$\beta$ & Diffusivity [$^\circ$C$^2$/h] & 0.1 \\
\hline
$\Delta x$ & Temperature deadband width [$^\circ$C] & 0.5\\
$a$ & Weighting function coefficient & $-1$ \\
$k_0$ & Control gain & 7.5\\
\hline
\end{tabular}
\end{table}
The diffusivity of the Fokker--Planck equations is taken from \cite{ZLZL:2019}, which is obtained by applying the algorithm developed in \cite{Moura:2014}. {The simulation is performed for a period of 24~hours.} The results related to power consumptions are all presented in quantities normalized by the maximal total power consumption of the population. The deadband width is set to $\Delta x = 0.5$~$^\circ$C. Initially, the temperatures of the TCLs are uniformly distributed over [20$-\Delta x /2$, 20$+\Delta x /2$]($^\circ$C), and their operational state is randomly generated with approximately 50\% in ON-state and 50\% in OFF-state.
\begin{figure}[htpb]
  \centering
  \includegraphics[scale=0.223]{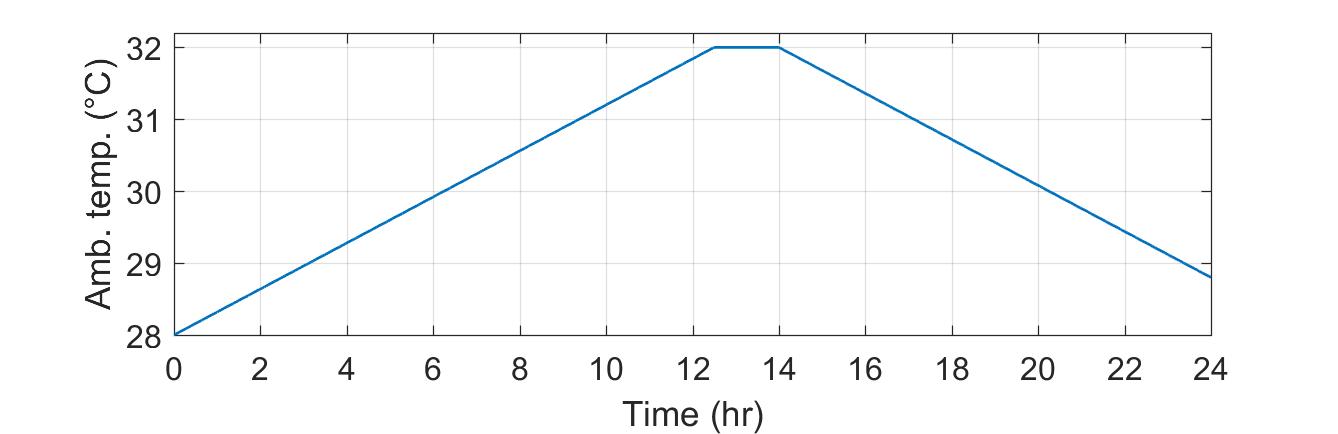}
  \caption{Ambient temperature.}\label{Fig: Tracking_ext}
\end{figure}

In the simulation, the forced switchings are randomly generated independently at the level of each TCL. Moreover, we note that by rewriting the weighting function in the output defined in \eqref{eq: output} as $ax+b = a(x-x_p)$, then $x_p$ can be interpreted as the set-point. For this reason, in aggregate power tracking control we can fix $y_d$ while varying $x_p$, which can be determined by the desired total power consumption of the population (see, e.g., \cite{Radaideh2019}). Furthermore, to avoid power consumption oscillations du to fast step set-point variations (\cite{Callaway:2009,Ihara1981,laparra2019,Radaideh2019,Totu:2017}), we consider a tracking control scheme where the reference trajectory, $x_p$, a smooth functions connecting the initial point at time $t_i$ to a desired point at time $t_f$. We choose then a polynomial of the following form taken from \cite{Levine:2009}:
\begin{equation}\label{eq: ref}
  x_p(t) = x_p(t_i) + \left(x_p(t_f)-x_p(t_i)\right)\tau^5(t)\sum_{l=0}^{4}a_l\tau^l(t), \, t \in [t_i, t_f],
\end{equation}
where $\tau(t)=(t-t_i)/(t_f - t_i)$. For a set-point control, the coefficients in \eqref{eq: ref} can be determined by imposing the initial and final conditions:
$
  \dot{x}_p(t_i)= \dot{x}_p(t_f)= \ddot{x}_p(t_i)= \ddot{x}_p(t_f) = \dddot{x}_p(t_i)= \dddot{x}_p(t_f) = 0,
$
which yields $a_0 = 126$, $a_1 = ?420$, $a_2 = 540$, $a_3 = ?315$, and $a_4 = 70$. The reference temperature trajectory used in the simulation is shown in Fig.~\ref{Fig: references}.
\begin{figure}[htpb]
  \centering
  \includegraphics[scale=0.225]{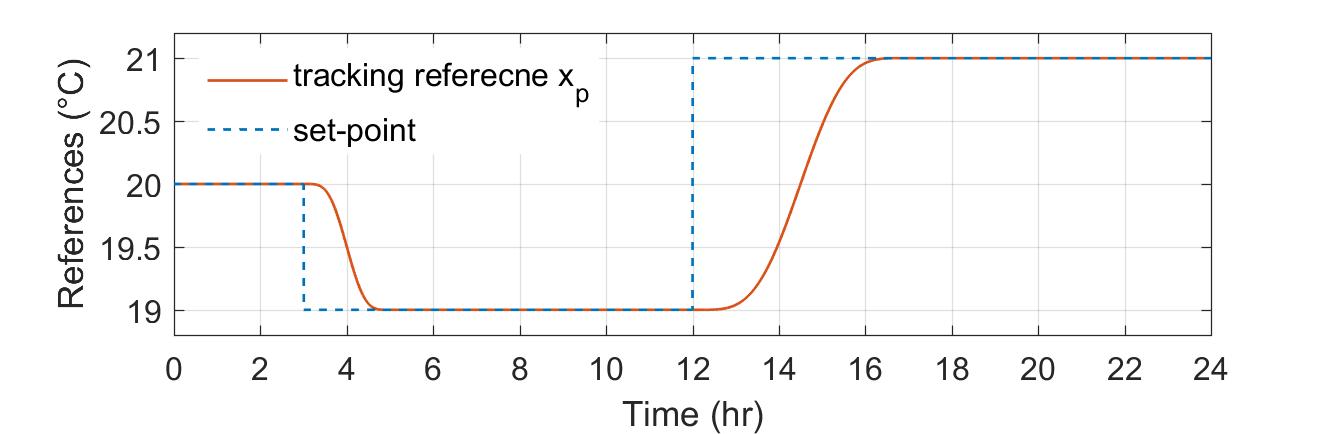}
  \caption{Step set-point variations and smooth reference trajectory.
           }\label{Fig: references}
\end{figure}

The control signal is depicted in Fig.~\ref{Fig: Tracking_cntr}. Figure~\ref{Fig: Tracking_ref} shows the desired temperature trajectory $x_p$ and the one generated from the PDE control signal as
\begin{equation*}
  x_{\text{ref}}(t) = x_{\text{ref}}(0)+\int_{0}^{t}\dot{x}_{\text{ref}}(\tau)\diff \tau =  x_{\text{ref}}(0)+\int_{0}^{t}u(\tau)\diff\tau.
\end{equation*}
Note that in order to avoid the chattering induced by fast variations of $u(t)$, the reference sent to the TCLs is the average of the values over the last 10~iterations. Figure~\ref{Fig: Tracking_temp} illustrates the temperature evolution of 200 randomly selected TCLs. It can be seen that all the TCLs follow well the reference while respecting the temperature deadband. The variation of the aggregate power consumption of the whole population is shown in Fig.~\ref{Fig: Tracking_power}. The results show that the developed control scheme tracks well the desired reference with quite import variations.

\begin{figure}[htpb]
  \centering
  \subfigure[]{\label{Fig: Tracking_cntr}\includegraphics[scale=0.225]{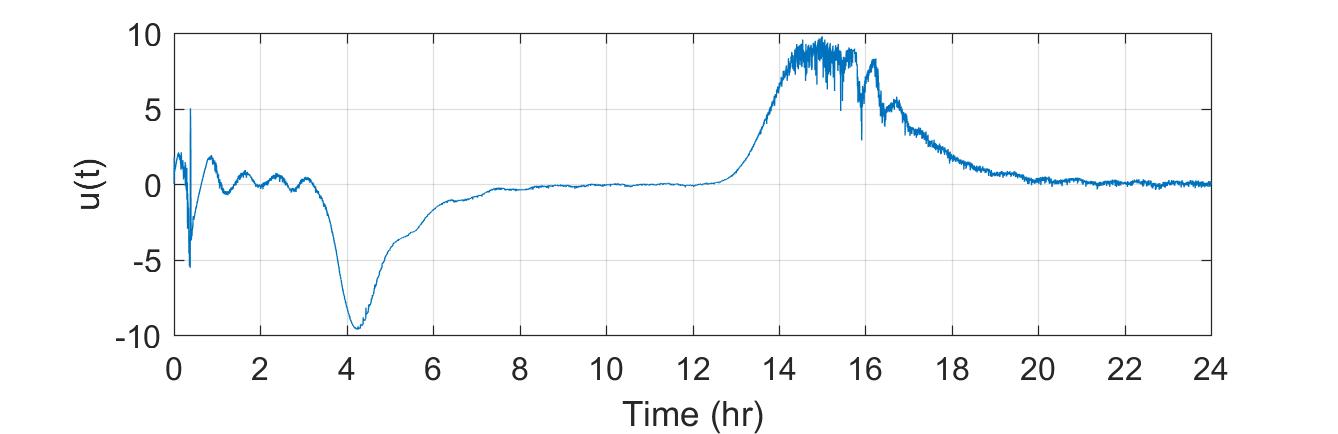}}
  \subfigure[]{\label{Fig: Tracking_ref}\includegraphics[scale=0.225]{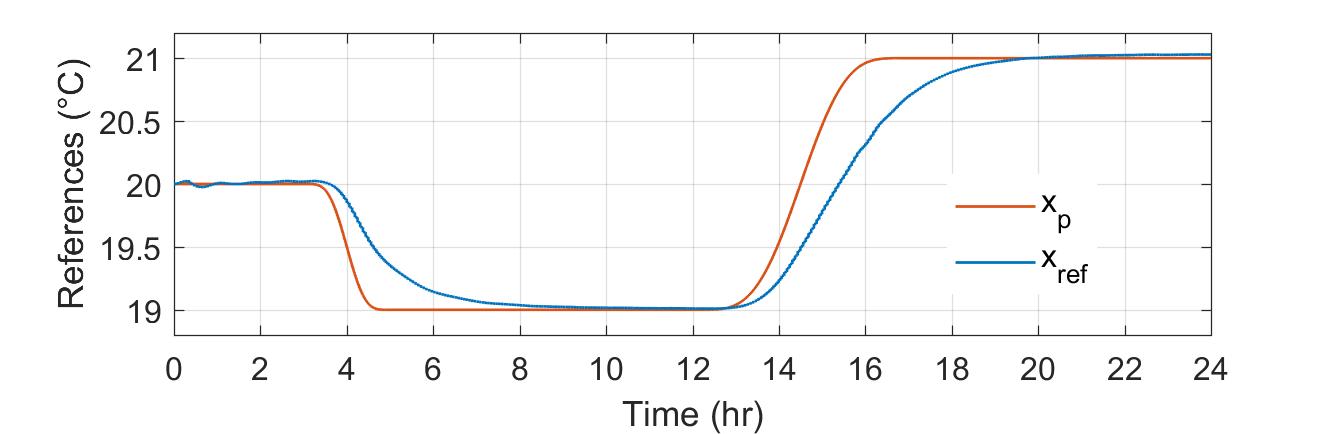}}
  \subfigure[]{\label{Fig: Tracking_temp}\includegraphics[scale=0.225]{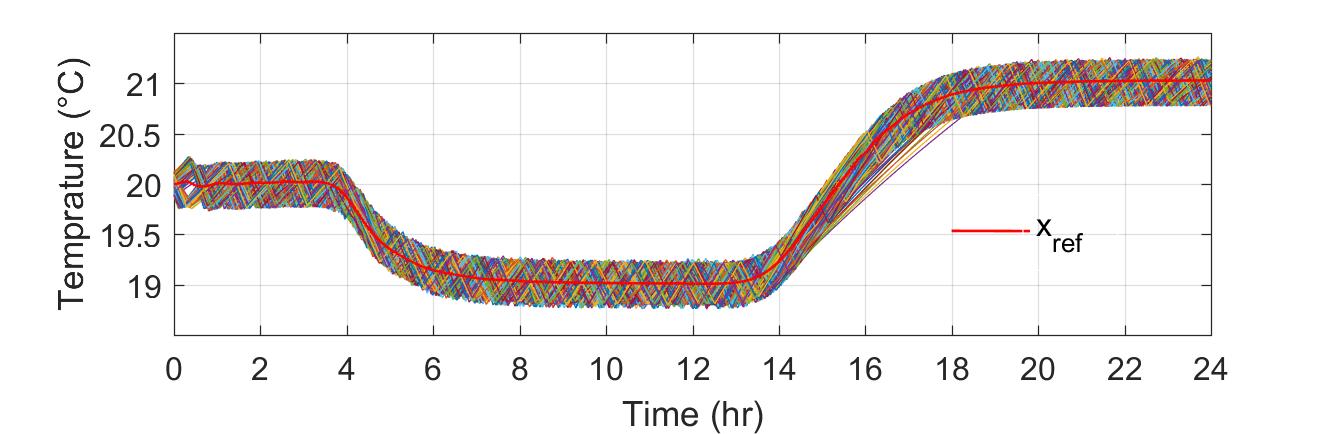}}
  \subfigure[]{\label{Fig: Tracking_power}\includegraphics[scale=0.225]{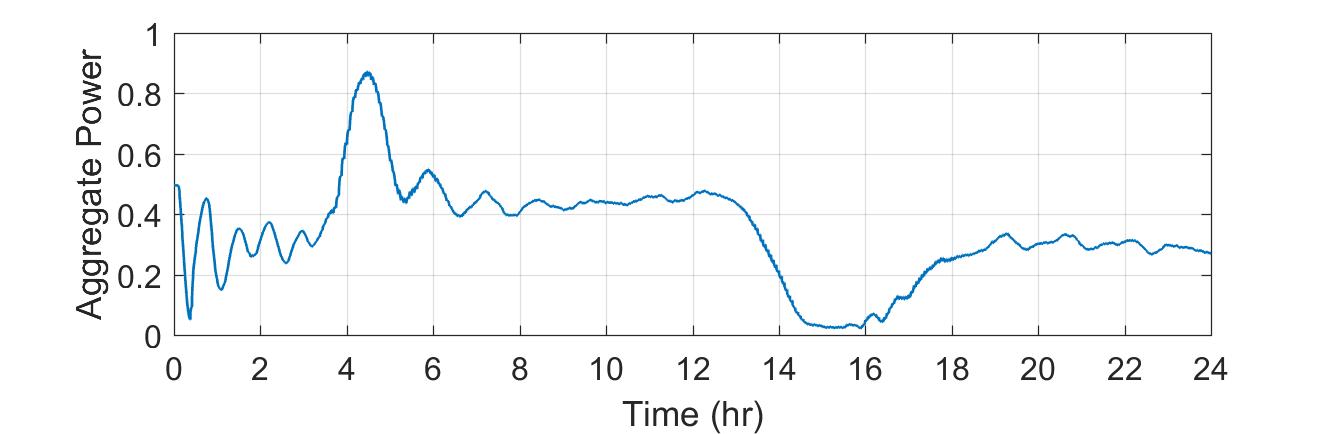}}
  \caption{Dynamics of the TCL population:
           (a) PDE control signal;
           (b) desired and generated references;
           (c) temperature evolutions;
           (d) aggregate power consumptions.
           }
\end{figure}

It is worth noting that since the feedback linearization--based control can cope with parameter variations in a straightforward manner, {there is no any concern about the simultaneous variation of set-point and ambient temperature in the considered setting, which is indeed hard to handle with most of the methods reported in the existing literature.} Moreover, the purpose of the simulation study presented in this work is to illustrate the essential behaviours of the developed control algorithm, and the emphasis is not put on the performance, although it can be improved by a finer tuning.

\section{Concluding remarks}\label{Sec: Concluding Remarks}
We have developed in the present work an input-output linearization--based scheme for aggregate power tracking control of heterogeneous populations of  thermostatically controlled load (TCLs) governed by a pair of coupled Fokker--Planck equations. As the thermostat--based deadband control is used in the operation of individual TCLs, the considered problem involves moving boundaries varying with the reference signals. It should be noted that a more generic setting is the case where the upper and lower bounds of the thermostat may vary with different rates. As the theoretical development, in particular the well-posedness assessment, of this type of systems will involve a very heavy mathematical analysis, it is beyond the scope of this work and will be addressed in a separate work. Other subjects, such as the control of TCL populations with the dynamics of the TCLs described by second-order dynamical models in the framework of partial differential equations, will also be considered in our future work.

\section{Appendix}
\subsection{Notations on function spaces}\label{Appendix: notation}
Let $D$ be {an open or closed} domain in $\mathbb{R}^i(i=1,2)$. $C^j(D) (j=1,2,...)$ consists of all continuous functions $u$ having continuous derivatives over $D$ up to order $j$ inclusively. $C^\infty(D)$ consists of all functions $u$ belonging to $ C^j(D)$ for all $j=1,2,...$. $C_0^\infty(D)$ consists of all $C^\infty$-functions defined on $D$ and having a compact support in $D$.

For $T\in(0,+\infty]$ and $a,b\in\mathbb{R} $ with $a<b$, let $ Q=(a,b)\times (0,T)$. $C^{2,1}(\overline{Q})$ consists of all continuous functions $u(x,t)$ having continuous derivatives $u_x,u_{xx},u_t$ over $\overline{Q}$. For a nonnegative integer $m$, we also use $D^m_xu$ to denote the $m$th-order derivative of a function $u$ w.r.t. its argument $x$. Let $l$ be a nonintegral positive number. $\mathcal{H}^{l,l/2}(\overline{Q})$ denotes the Banach space of functions $u(x,t)$ that are continuous over $\overline{Q}$, together with all derivatives of the form $D^r_tD^s_x$ for $2r+s<l$, and have a finite norm
    $|u|^{(l)}_{Q}:= \langle u\rangle^{(l)}_{Q}+\sum_{j=0}^{[l]}\langle u\rangle^{(j)}_{Q}$, 
  where
\begin{align}
    \langle u\rangle^{(0)}_{Q}:=&|u|^{(0)}_{Q}:=\max_{\overline{Q}}|u|,
    \langle u\rangle^{(j)}_{Q}:= \sum\limits _{(2r+s=j)}|D^r_tD^s_xu|^{(0)}_{Q},
    \langle u\rangle^{(l)}_{Q}:=\langle u\rangle^{(l)}_{x,Q}+\langle u\rangle^{(l/2)}_{t,Q}, \notag\\
    \langle u\rangle^{(l)}_{x,Q}:=& \sum\limits _{(2r+s=[l])}\langle D^r_tD^s_xu\rangle^{(l-[l])}_{x,Q},
    \langle u\rangle^{(l/2)}_{t,Q}:=  \sum\limits _{0<l-2r-s<2}\langle D^r_tD^s_xu\rangle^{(\frac{l-2r-s}{2})}_{t,Q}, \notag\\
    \langle u\rangle^{(\tau)}_{t,Q}:=&\sup_{(x,t),(x,t')\in \overline{Q},|t-t'|\leq \rho_0}\frac{|u(x,t)-u(x,t')|}{|t-t'|^{\tau}},\ 0<\tau <1,\notag\\
     \langle u\rangle^{(\tau)}_{x,Q}:=&\sup_{t\in (0,T)}\sup \rho^{-\tau}\text{osc}_x\left\{u(x,t);\Omega_\rho^i\right\},\ 0<\tau <1.\notag
\end{align}
In the last line, the second supremum is taken over all connected components $ \Omega_\rho^i$ of all $\Omega_\rho$ with $\rho{}{\leq} \rho_0$, while $\text{osc}_x\{u(x,t);\Omega_\rho^i\}$ is the oscillation of $u(x,t)$ on $\Omega_\rho^i$, i.e., $\text{osc}_x\{u(x,t);\Omega_\rho^i\} =\{\text{vrai}\sup_{x\in \Omega_\rho^i}u(x,t)-\text{vrai}\inf_{x\in \Omega_\rho^i}u(x,t)\}$, {}{$\rho_0=b-a$}, $\Omega_\rho=K_\rho\cap (a,b) $, and $K_\rho$ is an arbitrary open interval in $(-\infty,+\infty)$ of length $2\rho$.

 For $i=1,2$, $L^i(a,b)$ consists of all measurable functions $u$ defined over $(a,b)$ with $\|u\|_{L^i(a,b)}:=\left(\int_a^b|u(x)|^i\diff x\right)^{\frac{1}{i}}<+\infty$.
{$L^{1}(Q)$ consists of all measurable functions defined over $Q$ with $
\|u\|_{L^1(Q)}:=\int_0^T\int_a^b|u(x,t)|\diff x\diff t<+\infty
$. }
  $L^{\infty}(Q)$ consists of all measurable functions defined over $Q$ with $
\|f\|_{L^{\infty}(Q)}
:=\text{vrai}\sup_{(x,t)\in Q}|f(x,t)|<+\infty.
$

 Without a confusion, we denote $\|\cdot\|_{L^\infty(Q_T)}$ ( or $\|\cdot\|_{L^\infty(0,T)}$, or $\|\cdot\|_{L^\infty(0,1)}$) and $\|\cdot\|_{L^i(0,1)}$ ($i=1,2$)  by $\|\cdot\|_\infty$ and $\|\cdot\|_{i}$, respectively.

\section{Proof of Proposition \ref{well-posedness}\label{Appendix: well-posedness proof}}

\textbf{Step 1:} We establish the existence of a (classical) solution for approximating equations. For any {$T\in \mathbb{R}_{+}$}, we consider two sequences of parabolic equations over $Q_T$:
\begin{subequations}\label{+35}
\begin{align}
 \widetilde{w}_{nt} =  \widehat{\beta} \widetilde{w}_{nzz} -\left((\widehat{\alpha}_{1}-2G_{\widetilde{w}_{n-1}})\widetilde{w}_n\right)_z
   +\widetilde{\delta}_{n} ,  ~&(z,t)\in Q_T, \label{+35a}\\
    \widehat{\beta} \widetilde{w}_{nz}(1,t)-(\widehat{\alpha}_1(1)-2G_{\widetilde{w}_{n-1}})\widetilde{w}_n(1,t)= {\widehat{\overline{\sigma}}_n(t)}, ~&t\in [0,T], \label{+35b}\\
    \widehat{\beta }\widetilde{w}_{nz}(0,t)-(\widehat{\alpha}_1(0)-2G_{\widetilde{w}_{n-1}})\widetilde{w}_n(0,t)= {\widehat{\underline{\sigma}}_n(t)}, ~&t\in  [0,T], \label{+35c}\\
     \widetilde{w}_{n}(z,0)=\widetilde{w}^0(z),~&z\in(0,1), \label{+35d}
\end{align}
\end{subequations}
\vspace{-18pt}
\begin{subequations}\label{+v}
\begin{align}
  \widetilde{v}_{nt}=                                          \widehat{\beta} \widetilde{v}_{nzz}-\left((\widehat{\alpha}_{0}-2G_{\widetilde{w}_{n}})\widetilde{v}_{n}\right)_z
                                         -\widetilde{\delta}_{n} , ~&(z,t)\in Q_T, \label{+35a'}\\
                                         \widehat{\beta} \widetilde{v}_{nz}(1,t)-(\widehat{\alpha}_0(1)-2G_{\widetilde{w}_{n}}(t))\widetilde{v}_{n}(1,t)=- {\widehat{\overline{\sigma}}_n(t)}, ~&t\in  [0,T],\label{+35b'}\\
       \widehat{\beta} \widetilde{v}_{nz}(0,t)-(\widehat{\alpha}_0(0)-2G_{\widetilde{w}_{n}}(t))\widetilde{v}_{n}(0,t)=- {\widehat{\underline{\sigma}}_n(t)} ~&t\in  [0,T],\label{+35c'}\\
       \widetilde{v}_{n}(z,0)=\widetilde{v}^0(z),~&z\in(0,1),\label{+35d'}
  \end{align}
\end{subequations}
where
\begin{align*}
G_{\widetilde{w}_{n}}(t)=&-\dfrac{\displaystyle \widehat{\beta} (\widetilde{w}_{n}(1,t)-\widetilde{w}_{n}(0,t))+\frac{\eta}{{\widehat{a}}P}({k_0}y_d+\dot{y}_d)}{ {\displaystyle\int_{0}^{1}\widetilde{w}_n \diff z}}+\dfrac{\displaystyle\int_{0}^{1}{ {\left(k_0z+\frac{k_0\widehat{b}}{\widehat{a}}+\widehat{\alpha}_1+\Delta x\dot{b}\right)}}\widetilde{w}_{n} \diff z}{{ \displaystyle\int_{0}^{1}\widetilde{w}_{n} \diff z}},\forall n\geq 0,
\end{align*}
with $
  \widetilde{w}_0(z,t):= \widetilde{w}^0(z)$, ${\frac{\widehat{b}}{\widehat{a}}=\frac{1}{\Delta x}\left(\underline{x}+\frac{b}{a}\right)}$, $ \{\widetilde{\delta}_n\}$ is a sequence of functions on $\overline{Q}_{T}$ and {$ \{ \widehat{\overline{\sigma}}_n\},\{\widehat{\underline{\sigma}}_n\}$ are sequences of functions on $[0,T]$} satisfying:
\begin{enumerate}
\item[$\bullet$] $  \widehat{\overline{\sigma}}_n,\widehat{\underline{\sigma}}_n\in C^1( [0,T])$, $\widehat{\overline{\sigma}}_n(0)=\widehat{\overline{\sigma}}(0)$, $\widehat{\underline{\sigma}}_n(0)=\widehat{\underline{\sigma}}(0),\forall n\geq 1$;
  \item[$\bullet$] $\| \widehat{\overline{\sigma}}_n\|_{L^\infty(0,T)}\leq \widehat{\overline{\sigma}}_0, \| \widehat{\underline{\sigma}}_n\|_{L^\infty(0,T)}\leq \widehat{\underline{\sigma}}_0,\forall n\geq 1$, where $\widehat{\overline{\sigma}}_0,\widehat{\underline{\sigma}}_0$ are positive constants;
\item[$\bullet$] {$ \widehat{\underline{\sigma}}_n\rightarrow \widehat{\overline{\sigma}}$ and $ \widehat{\underline{\sigma}}_n\rightarrow \widehat{\underline{\sigma}}$ in $L^1(0,T)$, as $n\rightarrow \infty$};
  \item[$\bullet$] $\widetilde{\delta}_{n}$ is H\"{o}lder continuous in $z$ with exponent $\theta$ and $C^1$-continuous in $t$ over $\overline{Q}_T$;
    \item[$\bullet$] $\|\widetilde{\delta}_n\|_{L^\infty(Q_T)}\leq \widehat{\delta}_0,\forall n\geq 1$, where $\widehat{\delta}_0$ is a positive constant;
    \item[$\bullet$] {$\widetilde{\delta}_{n}\rightarrow \widetilde{\delta}$  in $L^1(Q_T)$}, as $n\rightarrow \infty$;
    \item[$\bullet$]  $|\int_0^1\widetilde{w}^0(z)\diff z +\int_{0}^t\int_0^1\widetilde{\delta}_n(z,s)\diff z\diff s+ \int_{0}^t(\widehat{\overline{\sigma}}_n (s)-\widehat{\underline{\sigma}}_n (s))\diff s|\geq \frac{\widetilde{\delta}_0}{4}>0,\forall t\in  [0,T]$, $\forall n\geq 1$, where $ \widetilde{\delta}_0:=\frac{{\delta}_0}{\Delta x}$.
\end{enumerate}

We re-write \eqref{+35} as below:
 \begin{subequations}\label{+35'}
\begin{align}
  \widetilde{ w}_{nt}(z,t)-\widehat{\beta} \widetilde{ w}_{nzz} + b_n(z,t,\widetilde{ w}_n,\widetilde{ w}_{nz})=0,  ~&(z,t)\in Q_T, \label{+35a'}\\
  \widehat{\beta} \widetilde{ w}_{nz}(1,t)+\psi_n(1,t,\widetilde{ w}_n)=0~&t\in  [0,T],\label{+35b'}\\
  \widehat{\beta} \widetilde{ w}_{nz}(0,t)-\psi_n(0,t,\widetilde{ w}_n)=0,~&t\in  [0,T],\label{+35c'}\\
                                         \widetilde{ w}_{n}(z,0)=\widetilde{w}^0(z),~&z\in (0,1),\label{+35d'}
\end{align}
\end{subequations}
where
$b_n(z,t,\widetilde{ w},p)=\widehat{\alpha}_{1z}\widetilde{ w}+(\widehat{\alpha}_1-G_{\widetilde{ w}_{n-1}}(t))p-\widetilde{\delta}_n(z,t)$ and $ \psi_n(z,t,\widetilde{ w})=(1-2z)(\widehat{\alpha}_1-G_{\widetilde{ w}_{n-1}}(t))\widetilde{w}{-z\widehat{\overline{\sigma}}_n(t)+(1-z)\widehat{\underline{\sigma}}_n(t)}$.

We  {show that \eqref{+35'} has a unique (classical) solution $ \widetilde{w}_{n}\in \mathcal {H}^{2+\theta,1+\frac{\theta}{2}}(\overline{Q}_{T})$ satisfying
\begin{align}\label{+34'''}
\max_{\overline{Q}_{T}}|\widetilde{w}|+\max_{{}{\overline{Q}_{T}}}|\widetilde{w}_{z}|+|\widetilde{w}|^{(2+\theta)}_{Q_{T}}\leq \widetilde{C}_{1n},
\end{align}
{where $\widetilde{C}_{1n}$ is a positive constant depending only on $\widehat{\beta}$, $P$, $\eta$, $k_0$, $\Delta x$, $\widehat{a}$, $\delta_0$, $\widehat{\delta}_0$, $\|y_d\|_\infty$, $\|\dot{y}_d\|_\infty$, $\|\underline{x}\|_\infty$, $\|\widehat{\alpha}_1\|_\infty$, $\|\widehat{\alpha}_{1z}\|_\infty$,  $\widehat{\overline{\sigma}}_0$, $ \widehat{\underline{\sigma}}_0$, $\|\widetilde{w}^0\|_\infty$, $\big|\int_{0}^{1}\widetilde{w}^{0}\diff z\big|$}, $|\widetilde{w}^0|^{(2)}_{(0,1)}$, $|\widetilde{w}^0|^{(2)}_{(0,1)}$ and $n$.

 Indeed, we proceed by induction. For $n=1$, by virtue of \textbf{(A1)}, \textbf{(A2)}, \textbf{(A3)} and \eqref{eq: ON on 0}, there exists a constant $g_0>0$ {depending only on $\widehat{\beta},P,\eta,k_0,\Delta x,\widehat{a}$, $\|y_d\|_\infty$, $\|\dot{y}_d\|_\infty$, $\|\underline{x}\|_\infty$, $\|\widehat{\alpha}_1\|_\infty$, $\|\widehat{\alpha}_{1z}\|_\infty,\|\widetilde{w}^0\|_\infty$ and $\big|\int_{0}^{1}\widetilde{w}^{0}\diff z\big|$} such that $|G_{\widetilde{w}^0}(t)|\leq g_0$ for all $t\in[0,T]$. Then for any $(z,t,\widetilde{w},p)\in[0,1]\times[0,T]\times\mathbb{R}\times\mathbb{R}$, we infer from Young's inequality that
\begin{align*}
-\widetilde{w}{b_1}(z,t,\widetilde{w},p)\leq &\bigg( \|\widehat{\alpha}_{1z}\|_{\infty}+\frac{\|\widehat{\alpha}_1\|_{\infty}+g_0}{2} \bigg)\widetilde{w}^2+\frac{\|\widehat{\alpha}_1\|_{\infty}+g_0}{2}p^2+\frac{\widehat{\delta}_0^2}{2}\notag\\
:=&C_{1,0}p^2+C_{1,1}\widetilde{w}^2+C_{1,2},\\
-\widetilde{w}{\psi_1}(z,t,\widetilde{w},p)\leq& C_{1,3}\widetilde{w}^2+C_{1,4},
\end{align*}
where $C_{1,3},C_{1,4}$ are positive constants depending only on  $ g_0,\|\widehat{\alpha}_1\|_{\infty},\widehat{\overline{\sigma}}_0, \widehat{\underline{\sigma}}_0$.

By the compatibility condition \eqref{09021} and the definition of $\widetilde{w}^0$, it follows that
\begin{align}\label{-29}
  \widehat{\beta} \widetilde{w}^0_{z}(1)+\psi_1(1,0,\widetilde{w}^0(1))=
  \widehat{\beta} \widetilde{w}^0_{z}(0)-\psi_1(0,0,\widetilde{w}^0(0))=0.
\end{align}
{Then all the structural conditions of \cite[Theorem 7.4, Chap. V]{Ladyzenskaja:1968} are fulfilled. Therefore,} \eqref{+35'} has a unique (classical) solution $\widetilde{w}_1\in \mathcal {H}^{2+\theta,1+\frac{\theta}{2}}(\overline{Q}_{T})$. Moreover, by \cite[Theorem 7.3, 7.2, Chap. V]{Ladyzenskaja:1968}, we have
\begin{align*}
\max_{\overline{Q}_{T}}|\widetilde{w}_1|\leq& \lambda_{1,1}\e^{\lambda_1T}\max\{\sqrt{C_{1,2}}, \sqrt{C_{1,4}}, \widetilde{w}^0(1), \widetilde{w}^0(0)\}:=m_1,
\max_{\overline{Q}_{T}}|\widetilde{w}_{1z}|\leq M_{1},
\end{align*}
{where $ \lambda_{1,1},\lambda_1$ depend only on $\widehat{\beta}, C_{1,0}, C_{1,1}$, and $C_{1,3}$, $ M_{1}$ depends only on $ \widehat{\beta}$, $m_1$, $g_0$, $\|\widehat{\alpha}_1\|_\infty$, $\|\widehat{\alpha}_{1z}\|_\infty$ and $|\widetilde{w}^0|^{(2)}_{(0,1)}$.
Furthermore, by \cite[Theorem 5.4, Chap. V]{Ladyzenskaja:1968}, we have
$|\widetilde{w}_1|^{(2+\theta)}_{Q_{T}}\leq \overline{M}_1$,
 where $ \overline{M}_1$ depends only on $ M_1$, $m_1$, $\theta$, $\widehat{\beta}$, $\widehat{\delta}_0$ and $|\widetilde{w}^0|^{(2)}_{(0,1)}$.}

Note that
\begin{subequations}\label{-32}
\begin{align}
    \widehat{\beta} \widetilde{w}_{1z}(1,0)+\psi_1(1,0,\widetilde{w}_1(1,0))=
  \widehat{\beta} \widetilde{w}_{1z}(0,0)-\psi_1(0,0,\widetilde{w}_1(0,0))=0,\\
  \widetilde{w}_{1z}(1,0)=\widetilde{w}^0_{z}(1),\widetilde{w}_{1z}(0,0)=\widetilde{w}^0_{z}(0),
  \widehat{\overline{\sigma}}_1(0)=\widehat{\overline{\sigma}}(0),\widehat{\underline{\sigma}}_1(0)=\widehat{\underline{\sigma}}(0).
\end{align}
\end{subequations}
 By the definition of $ \psi_1$,  \eqref{-29} and \eqref{-32}, we have $G_{\widetilde{ w}_{1}}(0)=G_{\widetilde{ w}_{0}}(0)$, which implies that
 \begin{align*}
   \widehat{\beta} \widetilde{w}^0_{z}(1)-(\widehat{\alpha}_1(1)-G_{\widetilde{ w}_{1}}(0))\widetilde{w}^0(1)+\widehat{\overline{\sigma}}_1(0)=0,
        \widehat{\beta} \widetilde{w}^0_{z}(0)-(\widehat{\alpha}_1(0)-G_{\widetilde{ w}_{1}}(0))\widetilde{w}^0(0)-\widehat{\underline{\sigma}}_1(0)& =0.
  \end{align*}

   {Let $k>1$ be an integer. Assuming that for $n=k$, \eqref{+35'} has a unique (classical) solution $ \widetilde{w}_{k}\in \mathcal {H}^{2+\theta,1+\frac{\theta}{2}}(\overline{Q}_{T})$ satisfying \eqref{+34'''} and the following equalities:
    \begin{align}\label{-33}
   \widehat{\beta} \widetilde{w}^0_{z}(1)-(\widehat{\alpha}_1(1)-G_{\widetilde{ w}_{k}}(0))\widetilde{w}^0(1)+\widehat{\overline{\sigma}}_k(0)=0,
        \widehat{\beta} \widetilde{w}^0_{z}(0)-(\widehat{\alpha}_1(0)-G_{\widetilde{ w}_{k}}(0))\widetilde{w}^0(0)-\widehat{\underline{\sigma}}_k(0) =0,
  \end{align}

we need to prove the existence of a (classical) solution in the case that $n=k+1$.}

Noting that by \eqref{+35}, for all $t\in [0,T]$, it follows that
\begin{align}\label{-34}
 \int_{0}^{1} \widetilde{w}_k(z,t)\diff z = \int_{0}^{1} \widetilde{w}^0(z)\diff z + \int_{0}^t\int_{0}^{1} \widetilde{\delta}_k(z,s)\diff z\diff s+ \int_{0}^t\left(\widehat{\overline{\sigma}}_k (s)-\widehat{\underline{\sigma}}_k (s)\right)\diff s.
\end{align}
By the definition of ${G_{w_{k}}}(t)$, \eqref{+28},  \eqref{+34'''} and the choice of $ \widetilde{\delta}_n, \widehat{\overline{\sigma}}_n,\widehat{\underline{\sigma}}_n$, we have
$|{G_{\widetilde{w}_k}}(t)|\leq {g_k}$ for all $t\in [0,T]$,
where {$g_k$} is a positive constant depending only on $\widehat{\beta}$, $P$, $\eta,k_0$, $\Delta x$, $\widehat{a}$,  $\delta_0$, $\|y_d\|_\infty$, $\|\dot{y}_d\|_\infty$, $\|\underline{x}\|_\infty$, $\|\widehat{\alpha}_1\|_\infty$, $\|\widehat{\alpha}_{1z}\|_\infty$, $\|\widetilde{w}^0\|_\infty$ and $\|\widetilde{w}_k\|_\infty$.
{
Then for any $(z,t,\widetilde{w},p)\in[0,1]\times[0,T]\times\mathbb{R}\times\mathbb{R}$, we obtain
\begin{subequations}\label{c-n}
\begin{align}
-\widetilde{w}{b_{k+1}}(z,t,\widetilde{w},p)\leq &C_{k+1,0}p^2+C_{k+1,1}\widetilde{w}^2+C_{k+1,2},\\
-\widetilde{w}{\psi_{k+1}}(z,t,\widetilde{w},p)\leq &C_{k+1,3}\widetilde{w}^2+C_{k+1,4},
\end{align}
 \end{subequations}
where $C_{k+1,0},C_{k+1,1},C_{k+1,2},C_{k+1,3},C_{k+1,4}$ are positive constants depending only on $g_k$, $\|\widehat{\alpha}_1\|_{\infty}$, $\|\widehat{\alpha}_{1z}\|_{\infty}$, $\widehat{\overline{\sigma}}_0$, $\widehat{\underline{\sigma}}_0$, $\widehat{\delta}_0$.

 Noting the fact that $\widehat{\overline{\sigma}}_{k+1}(0)=\widehat{\overline{\sigma}}_{k}(0)=\widehat{\overline{\sigma}}(0),\widehat{\underline{\sigma}}_{k+1}(0)=\widehat{\underline{\sigma}}_{k}(0)=\widehat{\underline{\sigma}}(0)$ and by \eqref{-33}, we find that the following compatibility conditions hold true:
   \begin{align*}
   \widehat{\beta} \widetilde{w}^0_{z}(1)-(\widehat{\alpha}_1(1)-G_{\widetilde{ w}_{k}}(0))\widetilde{w}^0(1)+\widehat{\overline{\sigma}}_{k+1}(0)=0,
        \widehat{\beta} \widetilde{w}^0_{z}(0)-(\widehat{\alpha}_1(0)-G_{\widetilde{ w}_{k}}(0))\widetilde{w}^0(0)-\widehat{\underline{\sigma}}_{k+1}(0) =0.
  \end{align*}
{Then, by \cite[Theorem 7.4, 7.3, 7.2, Theorem 5.4, Chap. V]{Ladyzenskaja:1968},} we conclude that \eqref{+35'} has a unique solution $ \widetilde{w}_{k+1}\in \mathcal {H}^{2+\theta,1+\frac{\theta}{2}}(\overline{Q}_{T})$ when {$n=k+1$}, having the estimate {in \eqref{+34'''}. Therefore, \eqref{+35} has a unique (classical) solution $ \widetilde{w}_{n}\in \mathcal {H}^{2+\theta,1+\frac{\theta}{2}}(\overline{Q}_{T})$ for every {$n\geq 1$}}.

{Since for any $n$ there exists $ \widetilde{w}_{n}$ satisfying \eqref{+35},} $G_{ \widetilde{w}_{n}}(t)$ is fixed for any fixed $n$. Thus, the existence of a unique (classical) solution $\widetilde{v}_n\in \mathcal {H}^{2+\theta,1+\frac{\theta}{2}}(\overline{Q}_{T})$ to \eqref{+v} is guaranteed by \cite[Theorem 7.4, Chap. V]{Ladyzenskaja:1968}. Moreover, $\widetilde{v}_n$ has a similar estimate as \eqref{+34'''}:
%
%
%
%
%
%
    \begin{align}\label{+122101}
\max_{\overline{Q}_{T}}|\widetilde{v}|+\max_{{}{\overline{Q}_{T}}}|\widetilde{v}_{z}|+|\widetilde{v}|^{(2+\theta)}_{Q_{T}}\leq \widetilde{C}_{2n},
\end{align}
where $\widetilde{C}_{2n}$ is a positive constant which may depend on $n$.

\textbf{Step 2:} We establish uniform estimates of $\widetilde{w}_n$ and $\widetilde{v}_{n}$ in $\mathcal {H}^{2+\theta,1+\frac{\theta}{2}}(\overline{Q}_{T})$, i.e., we prove that $\widetilde{C}_{1n}$ and $\widetilde{C}_{2n}$ are independent of $n$.
First, for any $n$ and any $t\in [0,T]$, we prove the following $L^1$-estimates:
\begin{subequations}\label{-38}
\begin{align}
\|\widetilde{w}_n(\cdot,t)\|_1  \leq&  \|\widetilde{w}^0\|_1 {+\displaystyle\int_0^t \widehat{\overline{\sigma} }_n (s)\text{sgn}(\widetilde{w}_n(1,s))\diff s}- {\displaystyle\int_0^t \widehat{\underline{\sigma}}_n(s)\text{sgn}(\widetilde{w}_n(0,s))\diff s}
+\displaystyle\int_0^t\int_{0}^{1}\widetilde{\delta}_n(z,s) \text{sgn}(\widetilde{w}_n(z,s)) \diff z\diff s,\label{-38a}\\
\|\widetilde{v}_n(\cdot,t)\|_1  \leq&  \|\widetilde{v}^0\|_1
{-\displaystyle\int_0^t \widehat{\overline{\sigma} }_n (s)\text{sgn}(\widetilde{v}_n(1,s))\diff s}
{+\displaystyle\int_0^t \widehat{\underline{\sigma}}_n(s)\text{sgn}(\widetilde{v}_n(0,s))\diff s}
-\displaystyle\int_0^t\int_{0}^{1}\widetilde{\delta}_n(z,s) \text{sgn}(\widetilde{v}_n(z,s)) \diff z\diff s.\label{-38b}
\end{align}
\end{subequations}
Indeed, for any $\varepsilon>0$, let
\begin{equation*}
\rho_\varepsilon(r):=\left\{\begin{aligned}
& |r|,\  |r|\geq \varepsilon\\
& -\frac{r^4}{8\varepsilon^3}+\frac{3r^2}{4\varepsilon}+\frac{3\varepsilon}{8},\  |r|< \varepsilon
\end{aligned}\right.,
\end{equation*}
which is $C^2$-continuous in $r$ and satisfies $\rho_\varepsilon(r)\geq |r|$, $|\rho_\varepsilon'(r)|\leq 1$ and $\rho_\varepsilon''(r)\geq 0$.
Multiplying $\rho_\varepsilon'(\widetilde{w}_n) $ to \eqref{+35}, and integrating by parts, we have
\begin{align}\label{-15}
\frac{\text{d}}{\text{d} t}\int_0^1\rho_\varepsilon(\widetilde{w}_n)\diff z=&-\widehat{\beta } \|\widetilde{w}_{nz}\sqrt{\rho_\varepsilon''(\widetilde{w}_n)}\|^2+\int_0^1 \left(\widehat{\alpha}_{1}-2G_{\widetilde{w}_{n-1}}\right)\rho_\varepsilon''(\widetilde{w}_n)\widetilde{w}_n\widetilde{w}_{nz}\diff z
+\int_0^1 \widetilde{\delta}_n\rho_\varepsilon'(\widetilde{w}_n)\diff z\notag\\
&
+\widehat{\overline{\sigma}}_n(t)\rho_\varepsilon'(\widetilde{w}_n(1,t))
-\widehat{\underline{\sigma}}_n(t)\rho_\varepsilon'(\widetilde{w}_n(0,t)).
\end{align}
Noting that
\begin{align*}
\|\widetilde{w}_n\sqrt{\rho_\varepsilon''(\widetilde{w}_n)}\|^2
=&\int_{0}^{1}\widetilde{w}_n^2\rho_\varepsilon''(\widetilde{w}_n)\chi_{\{|\widetilde{w}_n|>\varepsilon\}}(z)\diff z+ \int_{0}^{1}\widetilde{w}_n^2\rho_\varepsilon''(\widetilde{w}_n)\chi_{\{|\widetilde{w}_n|\leq\varepsilon\}}(z)\diff z\notag\\
=&{\int_{0}^{1}\widetilde{w}_n^2\rho_\varepsilon''(\widetilde{w}_n)\chi_{\{|\widetilde{w}_n|\leq\varepsilon\}}(z)\diff z}\notag\\
=&\int_{0}^{1}\widetilde{w}_n^2\frac{3(\varepsilon^2-\widetilde{w}_n^2)}{2\varepsilon^3}\chi_{\{|\widetilde{w}_n|\leq\varepsilon\}}(z)\diff z \notag\\
\leq & \frac{3}{2}\varepsilon,
\end{align*}
it follows that
\begin{align}\label{-16}
\int_{0}^{1}\left(\widehat{\alpha}_{1}-2G_{\widetilde{w}_{n-1}}\right)\rho_\varepsilon''(\widetilde{w}_n)\widetilde{w}_n\widetilde{w}_{z}\diff z
\leq & \frac{1}{4\widehat{\beta}} \|\left(\widehat{\alpha}_{1}-2G_{\widetilde{w}_{n-1}}\right)\widetilde{w}_n\sqrt{\rho_\varepsilon''(\widetilde{w}_n)}\|^2
+\widehat{\beta}\|\widetilde{w}_{z}\sqrt{\rho_\varepsilon''(\widetilde{w}_n)}\|^2\notag\\
\leq &\frac{1}{4\widehat{\beta}} \|\widehat{\alpha}_{1}-2G_{\widetilde{w}_{n-1}}\|_{\infty}^2 \|\widetilde{w}_n\sqrt{\rho_\varepsilon''(\widetilde{w}_n)}\|^2
+\widehat{\beta}\|\widetilde{w}_{z}\sqrt{\rho_\varepsilon''(\widetilde{w}_n)}\|^2\notag\\
\leq & C_n\varepsilon+\widehat{\beta}\|\widetilde{w}_{nz}\sqrt{\rho_\varepsilon''(\widetilde{w}_n)}\|^2,
\end{align}
where $C_n$ is a positive constant depending only on $\|\widehat{\alpha}_{1}-2G_{\widetilde{w}_{n-1}}\|_{\infty}$ and $ \widehat{\beta}$.

By \eqref{-15} and \eqref{-16}, we have
\begin{align*}
\frac{\text{d} }{\text{d}  t}\int_0^1\rho_\varepsilon(\widetilde{w}_n)\diff z\leq&C_n\varepsilon
+\int_0^1 \widetilde{\delta}_n\rho_\varepsilon'(\widetilde{w}_n)\diff z+\widehat{\overline{\sigma}}_n(t)\rho_\varepsilon'(\widetilde{w}_n(1,t))
-\widehat{\underline{\sigma}}_n(t)\rho_\varepsilon'(\widetilde{w}_n(0,t)),
\end{align*}
which implies that
\begin{align*}
\int_0^1\rho_\varepsilon(\widetilde{w}_n(z,t))\diff z\leq& \int_0^1\rho_\varepsilon(\widetilde{w}^0(z))\diff z +C_n\varepsilon t
+\int_0^t\int_0^1 \widetilde{\delta}_n\rho_\varepsilon'(\widetilde{w}_n)\diff z\diff t
+\int_0^t\widehat{\overline{\sigma}}_n(s)\rho_\varepsilon'(\widetilde{w}_n(1,s))\diff s
 -\int_0^t\widehat{\underline{\sigma}}_n(s)\rho_\varepsilon'(\widetilde{w}_n(0,s))\diff s.
\end{align*}
Letting $\varepsilon\rightarrow 0$, we obtain
\begin{align*}
\int_0^1 \widetilde{w}_n(z,t) \diff z\leq& \int_0^1 \widetilde{w}^0(z) \diff z
+\int_0^t\int_0^1 \widetilde{\delta}_n\text{sgn}(\widetilde{w}_n)\diff z\diff t
+\int_0^t\widehat{\overline{\sigma}}_n(s)\text{sgn}(\widetilde{w}_n(1,s))\diff s
 -\int_0^t\widehat{\underline{\sigma}}_n(s)\text{sgn}(\widetilde{w}_n(0,s))\diff s,
\end{align*}
which, along with the variable transformation $z=\frac{x-\underline{x}}{\Delta x}$, gives \eqref{-38a}. \eqref{-38b} can be obtained in the same way.

Now by \eqref{-38a} and \eqref{-38b}, for any $n\geq 1$ and any $ t\in [0,T]$, it follows that
\begin{align*}
\|\widetilde{w}_n(\cdot,t)\|_1  \leq&  \|\widetilde{w}^0\|_1 {+\displaystyle\int_0^t| \widehat{\overline{\sigma} }_n (s)|\diff s}+ {\displaystyle\int_0^t |\widehat{\underline{\sigma}}_n(s)|\diff s}+\displaystyle\int_0^t\int_{0}^{1}|\widetilde{\delta}_n(z,s) | \diff z\diff s\\
\leq&  \|\widetilde{w}^0\|_1+\widehat{\overline{\sigma}}_0T+\widehat{\underline{\sigma}}_0T+\widehat{\delta}_0T<+\infty,\\
\|\widetilde{v}_n(\cdot,t)\|_1  \leq&  \|\widetilde{v}^0\|_1+\widehat{\overline{\sigma}}_0T+\widehat{\underline{\sigma}}_0T+\widehat{\delta}_0T<+\infty.
\end{align*}
 By the continuity of $\widetilde{w}_n$, we deduce that $\widetilde{w}_n$ is uniformly bounded in $n$ on $\overline{Q}_{T}$. By \eqref{-34}, \eqref{+28} and the choice of $ \widetilde{\delta}_n, \widehat{\overline{\sigma}}_n,\widehat{\underline{\sigma}}_n$, we have that $G_{\widetilde{w}_n}(t)$ is uniformly bounded in $n$ over $\overline{Q}_{T}$. Furthermore, $C_{n,0},C_{n,1},C_{n,2},C_{n,3},C_{n,4}$ in the structural conditions (see \eqref{c-n}) are independent of $n$. We conclude that {$\widetilde{C}_{1n}$ in \eqref{+34'''}} is essentially independent of $n$. Analogously, {$\widetilde{C}_{2n}$ in \eqref{+122101}} is independent of $n$.
%

\textbf{Step 3:}  We complete the proof by taking the limit of $(\widetilde{w}_n,\widetilde{v}_{n})$. Indeed, we have in Step 2 that $(\widetilde{w}_n,\widetilde{v}_n)$ is uniformly bounded in $\mathcal {H}^{2+\theta,1+\frac{\theta}{2}}(\overline{Q}_{T})\times \mathcal {H}^{2+\theta,1+\frac{\theta}{2}}(\overline{Q}_{T})$. Due to the fact that $\mathcal {H}^{2+\theta,1+\frac{\theta}{2}}(\overline{Q}_{T})\hookrightarrow\hookrightarrow C^{2,1}(\overline{Q}_{T})$, up to a subsequence, there exists $(\widetilde{w},\widetilde{v})\in C^{2,1}(\overline{Q}_{T})\times C^{2,1}(\overline{Q}_{T})$ satisfying $(\widetilde{w}_n,\widetilde{v}_n)\rightarrow (\widetilde{w},\widetilde{v})$ in $ C^{2,1}(\overline{Q}_{T})\times C^{2,1}(\overline{Q}_{T})$ as $n\rightarrow \infty$. Noting that $\widetilde{\delta}_n\rightarrow \widetilde{\delta}$ in $L^1( Q_T)$ and $ \widehat{\underline{\sigma}}_n\rightarrow \widehat{\overline{\sigma}},\widehat{\underline{\sigma}}_n\rightarrow \widehat{\underline{\sigma}}$ in $L^1(0,T)$,  we conclude that $(\widetilde{w},\widetilde{v})$ satisfies \eqref{eq: foced FPE ON'}, \eqref{eq: foced FPE OFF'}, \eqref{B1}, \eqref{B2}, \eqref{B3}, \eqref{B4} and \eqref{I1} in the sense of distribution.

\subsection{Proof of Theorem~\ref{uniqueness}\label{Appendix: uniqueness}}

Indeed, if $(w_1,v_1),(w_2,v_2)\in C^{2,1}(\overline{\Omega}_{T})\times C^{2,1}(\overline{\Omega}_{T})$ be two distributional solutions of \eqref{eq: foced FPE}, \eqref{eq: BC 2} and \eqref{initialwv}, using the transformation of variables $z=\frac{x-\underline{x}}{\Delta x}$, it's easy to see that $(\widetilde{w}_1,\widetilde{v}_1),(\widetilde{w}_2,\widetilde{v}_2)\in C^{2,1}(\overline{Q}_{T})\times C^{2,1}(\overline{Q}_{T})$ satisfy \eqref{eq: foced FPE ON'}, \eqref{eq: foced FPE OFF'}, \eqref{B1}, \eqref{B2}, \eqref{B3}, \eqref{B4} and \eqref{I1} in the sense of distribution. Furthermore, due to the regularity of $\widetilde{w}_1$ and $\widetilde{w}_2$, $W:=\widetilde{w}_1-\widetilde{w}_2$ satisfies the following equations pointwisely:
\begin{align*}
 &W_t= \left(\widehat{ \beta}W_z- \widehat{\alpha}_{1}W-2WG_{\widetilde{w}_2}(t) - 2 \widetilde{w}_1(G_{\widetilde{w}_1}(t)-G_{\widetilde{w}_2}(t) )\right)_z, \ (z,t)\in Q_T,\\
 &\widehat{\beta}W_z(1,t)- \widehat{\alpha}_{1}(1)W(1,t)-2W(1,t)G_{\widetilde{w}_2}(t) - 2\widetilde{w}_1(1,t)(G_{\widetilde{w}_1}(t)-G_{\widetilde{w}_2}(t) )=0, \ t\in [0,T],\\
 &\widehat{\beta}W_z(0,t)- \widehat{\alpha}_{1}(0)W(0,t)-2W(0,t)G_{\widetilde{w}_2}(t) - 2\widetilde{w}_1(0,t)(G_{\widetilde{w}_1}(t)-G_{\widetilde{w}_2}(t) )=0, \ t\in [0,T],\\
 &W(z,0)=0, \ z\in (0,1),
\end{align*}
where $G_{\widetilde{w}_i}(t) (i=1,2)$ is defined as in Appendix \ref{Appendix: well-posedness proof}.
Then we have
\begin{align}
 \dfrac{1}{2}\frac{\text{d}}{\text{d} t}\|W\|^2=&-\int_{0}^{1}W_z\left( \widehat{\beta}W_z- \widehat{\alpha}_{1}W-2WG_{\widetilde{w}_2}(t) \right)\diff z
 +\int_{0}^{1}2 W_z\widetilde{w}_1(G_{\widetilde{w}_1}(t)-G_{\widetilde{w}_2}(t) ) \diff z.\label{+47}
\end{align}
It's clear that there exists a positive constant $C_0$  such that
\begin{align}
|\alpha_1+2G_{\widetilde{w}_2}(t)|\leq C_0, \forall (z,t)\in \overline{Q}_{T}.\label{+48}
\end{align}
Noting \eqref{+28}, $\int_{0}^{1}\widetilde{w}_{1} \diff z=\int_{0}^{1}\widetilde{w}_{2} \diff z$ (see Proposition~\ref{prop. 8}), and
$\widetilde{w}_1(1,t)-\widetilde{w}_1(0,t)=\widetilde{w}_2(1,t)-\widetilde{w}_2(0,t)$ due to the assumptions on $w_1,w_2$, 
  there  exists a positive constant $C_1$  such that
\begin{align}
2\left|G_{\widetilde{w}_1}(t)-G_{\widetilde{w}_2}(t)\right|&=\dfrac{\displaystyle2\int_{0}^{1}{ {\left(k_0z+\frac{k_0\widehat{b}}{\widehat{a}}+\widehat{\alpha}_1+\Delta x\dot{b}\right)}}W \diff z}{{ \displaystyle\int_{0}^{1}\widetilde{w}_{2} \diff z}}\leq C_1\|W\|_1.\label{+49}
\end{align}

By \eqref{+47}, \eqref{+48}, \eqref{+49}, H\"{o}lder's inequality and Young's inequality, we have
\begin{align*}
\dfrac{1}{2}\frac{\text{d}}{ \text{d} t}\|W\|^2\leq&-\widehat{\beta}\|W_z\|^2+\frac{\varepsilon}{2}\|W_z\|^2
+\frac{C_0^2}{2\varepsilon}\|W\|^2+\frac{\varepsilon}{2}\|W_z\|^2+\frac{1}{2\varepsilon}\|\widetilde{w}_1\|^2 \cdot (C_1 \|W\|_1)^2\\
\leq & -\left(\widehat{\beta}-\varepsilon\right)\|W_z\|^2+\frac{1}{2\varepsilon}{\left(C_0^2+C_1^2\|\widetilde{w}_1\|^2 \right)\|W\|^2},
\\
:=&{ -\left(\widehat{\beta}-\varepsilon\right)\|W_z\|^2+C_2\|W\|^2},
\end{align*}
where $\varepsilon>0$ is a positive constant.
Choosing $\varepsilon\in (0,\widehat{\beta})$ and taking integration, we have
\begin{align*}
\|W(\cdot,t)\|^2\leq \|W(\cdot,0)\|^2\cdot {\e^{2C_2 t}}=0,
\end{align*}
based on which and by the continuity of ${W}$, it yields $W\equiv 0$ on $\overline{Q}_{T}$, which gives the uniqueness of the solution to $\widetilde{w}$-system.

Since the $\widetilde{v}$-system is linear, it is clear that $\widetilde{v}_1=\widetilde{v}_2 $ over $\overline{Q}_T$.

Finally, utilizing transformation of variables again, we obtain the desired result.


\end{document}